\documentclass{amsart}
\usepackage{tikz-cd, amssymb, mathtools, enumitem, bbm}
\usetikzlibrary{arrows.meta, patterns}

\newtheorem{theorem}{Theorem}
\newtheorem{corollary}[theorem]{Corollary}
\newtheorem{lemma}[theorem]{Lemma}
\newtheorem{proposition}[theorem]{Proposition}
\theoremstyle{definition}
\newtheorem{definition}[theorem]{Definition}
\newtheorem{criterion}[theorem]{Criterion}
\newtheorem{remark}[theorem]{Remark}

\DeclareMathOperator\Isom{Isom}
\DeclareMathOperator\supp{supp}
\DeclareMathOperator\out{Out}
\newcommand\kn{\protect\smash{{w_k}_{\protect\vphantom0}^{\protect\mathclap{\!\!-\!1\,\,\,}}w_n}} 
\newcommand\stackthree[3]{\begin{smallmatrix}
		\vphantom{0^x_x}\smash{\text{#1}}
	\\	\vphantom{0^x_x}\smash{\text{#2}}
	\\	\vphantom{0^x_x}\smash{\text{#3}}
	\end{smallmatrix}}

\usepackage[backref = page]{hyperref}
\usepackage{xpatch}
\makeatletter
\xapptocmd{\@sect}{\csname #1mark\endcsname{#7}}{}{}
\makeatother   
\usepackage{fancyhdr}
\renewcommand{\sectionmark}[1]{\markright{\thesection.\ #1}{}}

\usepackage[shortalphabetic, abbrev, backrefs]{amsrefs}

\newcommand\arxiv[1]{\href{https://arxiv.org/abs/#1}{arXiv:#1}}
\renewcommand\MR[1]{~\href{http://www.ams.org/mathscinet-getitem?mr=#1}{MR#1}}

\title[Exponential Decay in Weakly Hyperbolic Groups]{Linear Progress with Exponential Decay in Weakly Hyperbolic Groups}
\author{Matt Sunderland}
\date{2017-Oct-14}
\address{365 Fifth Avenue, New York, NY 10016}
\begin{document}

\begin{abstract}
A random walk $w_n$ 
on a separable, geodesic hyperbolic metric space $X$
converges to the boundary $\partial X$ with probability one
when the step distribution supports two independent loxodromics.
In particular, the random walk makes positive linear progress.
Progress is known to be linear with exponential decay when
(1) the step distribution has exponential tail and
(2) the action on $X$ is acylindrical.

We extend exponential decay to the non-acylindrical case.
\end{abstract}

\maketitle

\section{Introduction}\label{intro}
Suppose $X_n$ is a sequence of independent, 
identically distributed random variables,
each taking the values 1 and $-1$ 
with probabilities $p$ and $1 - p,$ respectively,
for some fixed $p \in [\frac12, 1].$
Since the $X_n$ have finite expectation $2p - 1,$
by the law of large numbers,
\(	\textstyle\frac1n (X_1 + \cdots + X_n) 
	\to 2p - 1
	\)
almost surely.
Setting $Z_n := X_1 + \cdots + X_n,$
one obtains the stochastic process 
that Woess
calls P\'olya's Walk (in one dimension)
\cite{woess00}*{\S{}I.1.A}.
In this process,
our walker starts at zero
and then wanders randomly on the real line,
each time taking either a unit-one step to the left with probability $p,$
or a unit-one step to the right with probability~$1 - p.$

P\'olya's Walk with $p > \frac12$
makes \emph{positive linear progress}, meaning that  
\[	\liminf_{n \to \infty}
	\textstyle \frac1n |Z_n| 
	> 0
	\text{ almost surely.}
	\]
Indeed,
\(	\frac1n |Z_n|
	\ge \frac1n Z_n
	\to 2p - 1 
	> 0
	\)
almost surely.
P\'olya's Walk with $p = \smash{\frac12},$
however,
has expected number of returns 
$\sum_n \mathbb  P(Z_n = 0) = \infty,$ since
\(	\mathbb P(Z_{2n} = 0) 
	= 2^{-2n} \binom{2n}n 
	\sim cn^{-0.5}
	\)
for some constant $c.$
Note that in P\'olya's Walk (with any $p$),
the walker can only return to zero at even times.
It follows that the almost sure number of returns is infinite as well,
and so P\'olya's Walk with $p = \smash{\frac12}$
does not make positive linear progress.
See, for example, 
Woess \cite{woess00}*{\S{}I.1.A}
or Lawler~\cite{lawler10}*{Theorem 1.1}.

We say that a real-valued stochastic process $Z_n$ 
makes \emph{linear progress with exponential decay} (Definition \ref{drift})
if we can find a sufficiently large constant $C > 0$ 
so that for all~$n,$
\begin{equation}
	\mathbb P(Z_n \le n/C) 
	\le Ce^{-n/C}
	.
	\label{intro:drift}
	\end{equation}
Note that if $Z_n$ satisfies~\eqref{intro:drift}, 
then so does $|Z_n|.$
P\'olya's Walk with $p > \frac12$ 
can be shown to make
linear progress with exponential decay
using standard results in concentration of measure;
see for example, 
\cite{maher12}*{Proposition A.1}.
Exponential decay can also be shown using the following argument,
which will generalize to the setting of Theorem \ref{thm}.

If a (not necessarily Markov) stochastic 
process $Z_n$
has constants $t, \epsilon > 0$ so that for all $n,$
\begin{equation}
	\mathbb E\big(e^{-t(Z_{n + 1} - Z_n)} \mid Z_n\big)
	\le 1 - \epsilon
	,
	\label{intro:prog}
	\end{equation}
then $Z_n$ makes linear progress with exponential decay
(Proposition \ref{prog:lem}).
We say a process with such $t, \epsilon$
makes \emph{uniformly positive progress}
(Criterion \ref{prog:def}).
For any random variable $Z,$
the moment generating function $f(t) = \mathbb E(e^{tZ})$
has the property that $f'(0) = \mathbb E(Z).$
Thus~\eqref{intro:prog} 
captures a notion of positive progress
independent of location.

Therefore, it remains only to show
that P\'olya's Walk with $p > \frac12$
makes uniformly positive progress.
Let $p \in (\frac12 ,1],$
and let $X_n$ be the sequence of independent, 
identically distributed random variables
taking the values 1 and $-1$ 
with probabilities $p$ and $1 - p,$ respectively.
The claim is that $Z_n := X_1 + \cdots X_n$
satisfies~\eqref{intro:prog}
for some positive $t, \epsilon.$
Since $X_{n+1}, Z_n$ are independent
and $Z_{n+1} - Z_n = X_{n+1},$
\[	\mathbb E\big(e^{-t(Z_{n+1} - Z_n)} \mid Z_n\big) 
	= \mathbb E\big(e^{-tX_{n+1}}\big)
	= pe^{-t} + (1-p)e^t
	\]
for all $n$ and $t.$
The right-hand side $f(t) := pe^{-t} + (1-p)e^t$
has value $f(0) = 1$
and derivative $f'(0) = 1 - 2p.$
Moreover, this derivative is negtative, since $p > \frac12.$
Hence $f(t) = 1 - \epsilon$ for some $t, \epsilon > 0,$ as desired.

\subsubsection*{Weakly hyperbolic groups}
Thinking of the steps $X_n$ in P\'olya's walk 
as elements of $\mathbb Z$ acting on the real line 
inspires a notion of random walks on non-abelian groups.
Let $G$ be a group acting by isometries on a metric space $(X, d_X),$
and let $\mu$ be a probability distribution on $G.$
Following the notation from Tiozzo \cite{tiozzo15}*{\S1},
we define the $(G, \mu)$-random walk $w_n$ on $X$
by independently drawing at each time $n$
an element $s_n$  from $G$ with distribution $\mu,$
and then defining the random variable $w_n := s_1 \cdots s_n.$
Fixing a basepoint $x_0$ in $X,$
we obtain $(w_nx_0)_{n \in \mathbb N},$
a stochastic process taking values in $X.$

The notions of positive linear progress 
and linear progress with exponential decay 
also generalize to non-abelian groups by putting $Z_n := d_X(x_0, w_nx_0).$
We say that the $(G, \mu)$-random walk on $X$
is \emph{weakly hyperbolic} 
if $X$ is separable, geodesic, and $\delta$-hyperbolic.
Note that $X$ need not be locally compact.
We say that the random walk is \emph{non-elementary}
if the support of $\mu$ generates a subgroup
containing two loxodromics with disjoint endpoints on $\partial X$
(Definition \ref{nel:def}).
Such loxodromics are sometimes called independent
\citelist{\cite{osin16}*{Theorem 1.1} \cite{maher-tiozzo}*{\S1}}.
Notably, the definition of non-elementary excludes P\'olya's Walk.

As shown by Maher and Tiozzo,
every non-elementary, weakly hyperbolic random walk 
makes positive linear progress.
Progress is linear with exponential decay, moreover,
when the support of $\mu$ 
is bounded in $X$~\cite{maher-tiozzo}*{Theorem 1.2}.

We say that $\mu$ has \emph{exponential tail} 
if there exists $\lambda > 0$ such that
\[	\sum\nolimits_{g \in G} 
	e^{\lambda d_X(x_0, gx_0)}
	\mu(g) 
	< \infty
	.
	\]
Mathieu and Sisto 
prove linear progress with exponential decay 
in the case of geodesic, hyperbolic $X$
and $\mu$ with exponential tail 
and with support generating a subgroup not virtually cyclic, 
acting acylindrically with unbounded orbits in $X$
\cite{mathieu-sisto}*{Theorem 9.1}.
If a group $G$ is not virtually cyclic and
acts acylindrically with unbounded orbits on a hyperbolic space,
then $G$ contains infinitely many loxodromics 
with disjoint endpoints
\cite{osin16}*{Theorem 1.1}.
In particular, $G$ is non-el\-e\-m\-ent\-ary (Definition~\ref{nel:def}).

The goal of this paper is to show the following.
\begin{theorem}\label{thm}
	Every non-elementary, weakly hyperbolic random walk
	with exponential tail
	makes linear progress with exponential decay.
	\end{theorem}

The result applies, for example, to the action of $\out(F_n)$
on the complex of free splittings $\mathcal{FS}(F_n).$
The action is shown to be non-elementary, weakly hyperbolic,
and non-acylindrical by Handel and Mosher
\cite{handel-mosher16}*{Theorem 1.4}.
Even when a group is acylindrically hyperbolic,
we may care about a non-acylindrical action 
on another hyperbolic space
because of the geometric insight 
garnered from the particular action and its loxodromics.

We may care about the rate of convergence,
instead of just whether something tends to zero, 
depending on the technical details of the random methods at hand.
For ex\-ample, Lubotzky, Maher, and Wu
use exponential decay in an essential way 
in their study of the Casson invariant 
of random Heegaard splittings~\cite{lubotzky-maher-wu}.

The idea of the proof is essentially the same as 
in P\'olya's Walk with $p > \frac12$:
prove that every non-elementary, weakly hyperbolic 
random walk with exponential tail
makes uniformly positive progress (Criterion \ref{prog:def}),
which we know implies linear progress with exponential decay
(Proposition \ref{prog:lem}).

\begin{proof}[Proof of Theorem \ref{thm}]
Working backwards, 
the goal is to find $C > 0$ so that for all~$n,$
\begin{equation}
	\mathbb P(d_X(x_0, w_nx_0) \le n/C) 
	\le Ce^{-n/C}
	.
	\label{thm:claim}
	\end{equation}
This proof is structured so that
each equation~\eqref{thm:claim}--\eqref{thm:drift}
is implied by the next.

We show that~\eqref{thm:claim}
follows from 
$a$-iterated linear progress with exponential decay,
i.e., the existence of $C > 0$ and integer $a > 0$ so that for all $n,$
\begin{equation}
	\mathbb P(d_X(x_0, w_{an}x_0) \le n/C) 
	\le Ce^{-n/C}
	,
	\label{thm:iter}
	\end{equation}
and exponential tail
in Corollary \ref{iter:cor}.

The proof of this corollary is based on the following intuition.
If the random walk satisfies~\eqref{thm:claim} 
at all times $n = ai$ ($i \in \mathbb Z$) 
and the random walk cannot wander too far
during the intervening times (by exponential tail),
then the random walk satisfies~\eqref{thm:claim} at all times~$n.$
The corollary is a special case
of the general fact 
that if $Y_n$ has exponential tail
and $Z_n$ makes linear progress with exponential decay,
then their sum also makes linear progress with exponential decay.

We show that linear progress with exponential decay
follows from uniformly positive progress
in Proposition \ref{prog:lem}.
Hence, to show~\eqref{thm:iter},
it suffices to find constants $b, \epsilon > 0$
such that for all $n,$
\begin{equation}
	\mathbb E\big(
		e^{-b(d_X(x_0, w_{an+a}x_0) - d_X(x_0, w_{an}x_0)} 
		\mid
		d_X(x_0, w_{an}x_0)\big) 
	\le 1 - \epsilon
	.
	\label{thm:prog}
	\end{equation}
Crucially,~\eqref{thm:prog} implies~\eqref{thm:iter} even though
the process $Z_n = d_X(x_0, w_{an}x_0)$ is not necessarily Markov.
The proof of the proposition is purely probabilistic.

Given $w, g$ in $G,$ define the horofunction
\(	\rho_g(w)
	= d_X(gx_0, wx_0)
	- d_X(x_0, gx_0) 
	.
	\)
We prove that~\eqref{thm:prog} holds 
as long as there are constants $a, b > 0$ so that 
\begin{equation}
	\sup\nolimits_{g \in G}
	\big(\mathbb E\big(e^{-b\rho_g(w_a)}\big)\big) 
	< 1
	\label{thm:horo}
	\end{equation}
in Lemma \ref{prog:cor}.
We say a random walk with such $a, b$
has \emph{uniformly positive horofunctions}
at time $a.$

Given $w, g$ in $G,$
define the Gromov product to be 
$(w, g)_1
	= \frac12 d_X(x_0, wx_0) 
	+ \frac12 d_X(x_0, gx_0) 
	- \frac12 d_X(wx_0, gx_0).$
In the hyperbolic setting,
one can think of $(w, g)_1$
as measuring the distance $w$ and $g$
(or rather a geodesic from $x_0$ to $wx_0$ 
and a geodesic from $x_0$ to $gx_0$)
fellow-travel.
We establish that~\eqref{thm:horo}
holds as long as there exist sufficiently large $d, k$ so that
\begin{equation}
	\sup\nolimits_{g, h \in G} 
	\mathbb P((w_kh, g)_1 \ge d) 
	\le 0.01
	\label{thm:back}
	\end{equation}
in Lemma \ref{horo:lem}.
We say a random walk with such $d, k$
has \emph{uniform shadow decay}.

The proofs of Lemmas \ref{prog:cor} and \ref{horo:lem}
are adapted from Mathieu and Sisto \cite{mathieu-sisto}*{Theorem 9.1}.
The key to the proof of Lemma \ref{prog:cor}
is the observation that $\mathbb E(\rho_{{w_{ai}}^{-1}}(w_a))$
equals 
\(	\mathbb E(
		d_X(w_{ai}x_0, w_{ai + a}x_0)
		- d_X(x_0, w_{ai + a}x_0)
		)
	\)
and hence
the horofunction $\rho_g(w_a)$ 
measures a notion of progress;
see Figures~\ref{prog:fig} and~\ref{prog:g}.
In the proof of Lemma \ref{horo:lem},
we essentially argue that
if you are unlikely to be in the shadow of $g,$
then you are unlikely to be in the horoball about~$g$
(regardless of hyperbolicity).

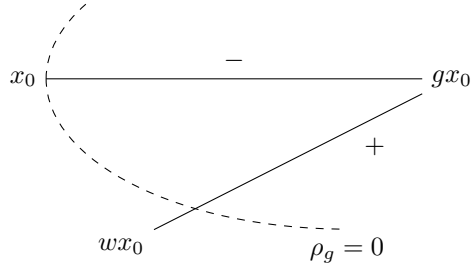
\begin{figure}
	\begin{tikzpicture}
		\draw[dashed, yscale = 0.5]
			(0, 0) arc (180: 270: 4) node[below] {$\rho_g = 0$}
			(0, 0) arc (180: 150: 4)
			;
		\draw
			(5, 0) node[right] (g) {$gx_0$}
			(0, 0) node[left] (x) {$x_0$}
			(1, -2) node[below] (w) {$wx_0$}
			(w) -- 
			node[below right, near end]
			{$+$} 
			(g)
			(x) -- 
			node[above]
			{$-$} 
			(g)
			;
		\end{tikzpicture}
	\caption{Stepping from $x_0$ to $wx_0,$
		the horofunction
		$\rho_g(w) = d_X(gx_0, wx_0) - d_X(x_0, gx_0)$
		measures progress away from the basepoint $gx_0$ 
		(regardless of hyperbolicity).
		In the proof of Theorem~\ref{thm},
		working backwards we show that \eqref{thm:iter}
		linear progress for the $a$-iterated random walk
		follows from \eqref{thm:prog} uniformly positive progress,
		which follows from \eqref{thm:horo} uniformly positive horofunctions.}
	\label{prog:fig}
	\end{figure}

We use the hyperbolicity of $X$
and the positive linear progress of $d_X(x_0, w_nx_0)$
to show that uniform shadow decay~\eqref{thm:back} 
follows from \emph{shadow decay}, i.e.,
\begin{equation}
	\adjustlimits \lim_{d \to \infty} \sup_{n, g} 
	\, \mathbb P((w_n, g)_1 \ge d) =
	\adjustlimits \lim_{d \to \infty} \sup_{n, g} 
	\, \mathbb P((w_n^{-1}, g)_1 \ge d) = 0
	,
	\label{thm:drift}
	\end{equation}
in Lemma \ref{back:lem}.
The proof uses only the Gromov four-point condition
and does not require the action of $G$ on $X$ to be acylindrical.
See Figure~\ref{back:claim}.

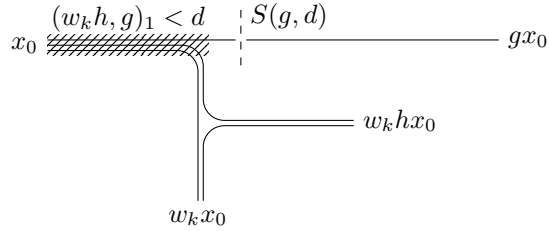
\begin{figure}
	\begin{tikzpicture}[rounded corners = 8]
		\fill[pattern = north east lines, rounded corners = 0] 
		  (0, -6pt) 
		  rectangle 
		  node[above, outer sep = 4pt] 
		  {$\smash{(w_kh, g)_1} < d$} 
		  (2cm + 4pt, 2pt) 
		  ;
		\draw (0, 0) -- (2.5, 0) ++(right: 4pt) -- (6, 0) node[right] {$gx_0$};
		\draw[yshift = -4] (0, 0) -- (2, 0) -- (2, -2) 
		  node[below] {$w_kx_0$};
		\draw[yshift = -2] 
		  (0, 0) node[left] {$x_0$} -- 
		  (2cm + 2pt, 0) -- ++(down: 1) -- 
		  ++(right: 2) node[right] {$w_khx_0$}
		  ++(down: 2pt) -- ++(left: 2) -- ++(down: 1)
		  ;
	\draw[dashed] 
		(2.5cm + 2pt, 0)
		++(up: 0.4) 
		-- node[right, pos = 0.10] {$S(g, d)$} 
		++(down: 0.8)
		;
		\end{tikzpicture}
	\caption{Uniform shadow decay \eqref{thm:back} means
		it is unlikely that $\smash{(w_kh, g)_1} \ge d$ for large $d, k$.
		Shadow decay \eqref{thm:drift}
		means it is unlikely that 
		$\smash{(w_n, g)_1} \ge d$ for large $d.$
		We derive \eqref{thm:back} from \eqref{thm:drift}
		using positive linear progress and hyperbolicity.
		}
	\label{back:claim}
	\end{figure}

Finally, both shadow decay~\eqref{thm:drift} and positive linear progress
are consequences of the convergence
of non-elementary, weakly hyperbolic random walks
(Theorem \ref{nel:thm} \cite{maher-tiozzo}*{1.1, 1.2, 5.3}).
\end{proof}

\section{Background}
\subsection{Metric hyperbolicity}\label{metric}
Given a metric space $(X, d_X)$ and three points $a, b, c$ in $X,$
the Gromov product of $b$ and $c$ with respect to $a$ is defined to be
\[	\textstyle
	(b, c)_a
	:= \frac12 d_X(a, b)
	+ \frac12 d_X(a, c)
	- \frac12 d_X(b, c)
	.
	\]
The metric space is said to be $\delta$-hyperbolic,
where $\delta$ is some fixed real number,
if every quadruple of points $a, b, c, d$ in $X$ 
satisfies the Gromov four-point condition 
\begin{equation}
	(b, c)_a \ge 
	\min \{(b, d)_a, (c, d)_a\} - \delta
	.
	\label{4ptX}
	\end{equation}
This $\delta$ cannot depend on the four points: 
the same fixed constant $\delta$ must work for all quadruples in $X.$
This condition can be shown to be equivalent to others, 
such as the slim triangles condition
\cite{bridson-haefliger}*{Propositions III.H\,1.17 and 1.22}.
When $\delta$ is obvious or not important,
we may say that $X$ is a hyperbolic metric space,
or colloquially, that $X$ has slim triangles.

Each hyperbolic space $X$ 
has an associated Gromov boundary $\partial X.$ 
If $X$ is locally compact, 
then both $\partial X$ and $\overline X = X \cup \partial X$ are compact. 
For a construction of the Gromov boundary, see for example 
Bridson and Haefliger~\cite{bridson-haefliger}*{\S III.H.3}.

In this paper we exclusively take the Gromov product 
of points in the orbit of a fixed basepoint $x_0$ in $X$ 
under a countable group of isometries $G \to \Isom X.$
Therefore, to simplify the notation, for each $w$ in $G,$ we denote
\[	|w| 
	:= d_X(x_0, wx_0)
	\]
(following Kaimanovich \cite{kaimanovich00}*{\S7.2}),
and given two more elements $g, h$ in $G$
we denote the Gromov product 
of $wx_0$ and $gx_0$ with respect to $hx_0,$
\[	\textstyle
	(w, g)_h 
	:= \frac12 d_X(wx_0, hx_0)
	+ \frac12 d_X(gx_0, hx_0)
	- \frac12 d_X(wx_0, gx_0)
	\]
(following Mathieu and Sisto \cite{mathieu-sisto}*{\S9,\,p.\,29}).

In the slim triangles setting,
the intuition is that 
the Gromov product measures distance fellow-travelled,
in the sense of Figure \ref{prod:fig}.
We sometimes say that $w$ and $g$ fellow-travel to mean that
all geodesics $[x_0, wx_0]$ and $[x_0, gx_0]$ in $X$ fellow-travel.
We use the notion of fellow-travelling for intuition only;
for proofs we appeal to the Gromov four-point condition directly.
\begin{figure}
	\begin{tikzpicture}[rounded corners = 8]
		\fill[pattern = north east lines, rounded corners = 0] 
		  (0, -4pt) rectangle 
		  node[above, outer sep = 4pt] {$\smash{(w, g)_1} < d$} 
		  (2cm + 2pt, 2pt) 
		  ;
		\draw (0, 0) node[left] {$x_0$}
		  -- (2.5, 0) ++(right: 4pt) -- (6, 0) node[right] {$gx_0$};
		\draw[yshift = -2] (0, 0) -- (2, 0) -- (2, -2) 
		  node[below] {\phantom{not in $S(g, d)$}$wx_0$ not in $S(g, d)$};
	\draw[dashed] 
		(2.5cm + 2pt, 0) ++(up: 0.7) 
		-- node[right, pos = 0.15] {$S(g, d)$} 
		++(down: 1.4)
		;
		\end{tikzpicture}
	\caption{When triangles are slim, the Gromov product 
		$(w, g)_1$ measures the distance that $w, g$ fellow-travel. 
		The shadow $S(g, d)$ contains $w$ 
		if and only if that distance $(w, g)_1 \ge d.$}
	\label{prod:fig}
	\end{figure}
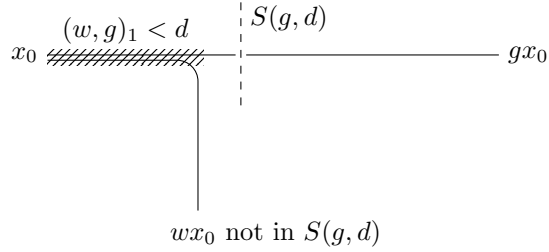


As in Figure \ref{4pt:fig}, the Gromov four-point condition~\eqref{4ptX} applied
to the points $x_0,$ $wx_0,$ $gx_0,$ $hx_0,$ implies that
\begin{equation}
	\begin{aligned}
			\text{either }
			(h, w)_1 &\le (w, g)_1 + \delta
		\\	\text{ or } 
			(h, g)_1 &\le (w, g)_1 + \delta
			.
		\end{aligned}
	\label{4ptG}
	\end{equation}
\begin{figure}
	\begin{tikzpicture}[rounded corners = 8]
		\fill[pattern = north east lines, rounded corners = 0] 
			(0, -8pt) 
			rectangle 
			node[above, outer sep = 4pt] 
			{$\smash{(w, g)_1} + \delta$} 
			(2cm + 4pt ,2pt) 
			;
		\draw (0, 0) -- (2, 0) -- (6, 0) node[right] {$gx_0$};
		\draw[yshift = -6] (0, 0) -- (2, 0) -- (2, -2) 
			node[below] {$wx_0$}
			;
		\node (wh) at (5, -1cm - 4pt) {$hx_0$};
		\draw[yshift = -4] (0, 0) 
			node[left] {$x_0$} -- (2cm + 2pt, 0) -- 
			+(down: 1) -- (wh)
			;
		\draw[yshift = -2] (wh) -- 
			+(up: 1cm + 2pt) -- (0, 0)
			;
		\node[yshift = -3] at (3.5, -0.5) {not both};
		\end{tikzpicture}
	\caption{The Gromov four point condition \eqref{4ptG}
		means that $h$ cannot fellow-travel both $w$ and $g$ 
		beyond distance $(w, g)_1 + \delta.$}
	\label{4pt:fig}
	\end{figure}
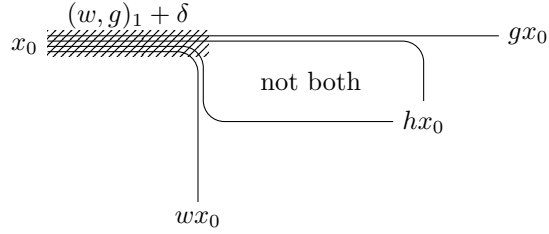

We will also use the following version of the triangle inequality 
in terms of the Gromov product:
\begin{equation}
	0 \le (w, g)_h \le 
	\min\{d_X(wx_0, hx_0), d_X(gx_0, hx_0)\}
	.
	\label{triangle}
	\end{equation}
This equation is equivalent to the triangle inequality,
and always holds regardless of whether or not $X$ is hyperbolic.

\subsubsection*{Loxodromics}
A map between metric spaces 
$f:(X,d_X)\to(Y,d_Y)$ 
is a $C$-{quasi-isometric embedding} if
\[	d_X(x, x')/C - C 
	\le d_Y(f(x), f(x')) 
	\le d_X(x, x')C + C
	\]
for all $x,$ $x'$ in $X,$ where $C$ 
is independent of the choice of $x,$ $x'.$
When the domain is $X = \mathbb Z,$
we call such a map a quasi-geodesic.
An isometry $g$ of $X$ is called a loxodromic
if for some (equivalently, all) $x_0$ in $X,$ 
the map $n \mapsto g^nx_0$ is a quasi-geodesic.

A loxodromic isometry of a hyperbolic space 
$g: \overline X \to \overline X$
fixes exactly two points, which are on the boundary.
These fixed points are $\lim_{n \to \infty} g^n x_0$
and $\lim_{n \to -\infty} g^n x_0$ regardless of $x_0,$
and are sometimes called the endpoints of the loxodromic.
See, for example,
Kapovich and Benakli~\cite{kapovich-benakli}*{Theorem 4.1}.

\subsubsection*{Acylindricality}
We do not use acylindricality in any of our proofs,
but the definition is included for completeness.
Suppose $G$ acts on $(X, d_X).$
The stabilizer (or isotropy group) of a point $x$ in $X$
is defined to be $G_x := \{g \in G : gx = x\}.$
The $r$-coarse stabilizer of $x$ is
\[	G_{x, r} 
	:= \{g \in G : d_X(x, gx) \le r\}
	.
	\]
We say that the action of $G$ is acylindrical
if every $r > 0$ has constants $d(r),$ $n(r)$
such that the $r$-coarse stabilizers 
of any two points of sufficient distance $d_X(x, y) \ge d$ 
have not too many group elements in common
\[	\#\, (G_{x, r} 
	\cap G_{y, r})
	\le n
	.
	\]

\subsubsection*{Classification of actions on hyperbolic spaces}\label{classify}
Consider a (possibly not proper) action of $G$ 
on the $\delta$-hyperbolic space $X$ by isometries.
As shown by Gromov \cite{gromov87}*{\S8.1--8.2},
the action of $G$ is either
\begin{enumerate}
	\item elliptic ($G$ has bounded orbits in $X$),
	\item parabolic ($G$ has unbounded orbits, but no loxodromics),
	\item lineal 
		($G$ has loxodromics, 
		all of which share the same two endpoints),
	\item quasi-parabolic
		($G$ has two loxodromics that share exactly one endpoint $\lambda,$
		and all loxodromics have an endpoint at $\lambda$), or
	\item non-elementary ($G$ has two loxodromics that share no endpoints).
	\end{enumerate}
Some sources define non-elementary 
to include quasi-parabolics \cite{osin16}*{\S3}.
However, if the action is acylindrical, 
then $G$ is neither parabolic nor quasi-parabolic
\cite{osin16}*{Theo\-rem 1.1}.
A group that admits a non-elementary,
acylindrical action on a hyperbolic space is called 
acylindrically hyperbolic.

\subsubsection*{Horofunctions}
Let $(X, d_X)$ be a metric space with basepoint $x_0$
and $G \to \Isom X$ 
a (not necessarily injective) group of isometries.
For each $g$ in $G,$
the horofunction $\rho_g: G \to \mathbb R$ is defined to be
\begin{equation}
	\rho_g(w)
	:= - d_X(x_0, gx_0) + d_X(gx_0, wx_0) 
	.
	\label{horo:eq}
	\end{equation}

Think of horofunctions as distance functions 
normalized to be zero at $x_0$;
see Figure \ref{horo:R}.
The normalization allows us 
to take limits of horofunctions.
Consider, for example, $G = X = \mathbb R.$
Then the horofunction $\rho_x(y) = -x + |y - x|$
approaches the identity function $y \mapsto y,$
as $x$ approaches $-\infty.$
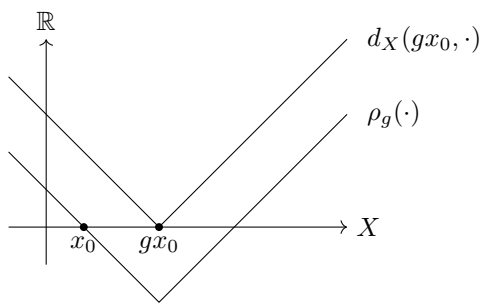
\begin{figure}
	\begin{tikzpicture}
		\draw[->] (0.5, -0.5) -- (0.5, 2.5) 
			node[above] {$\mathbb R$};
		\draw[->] (0, 0) -- (4.5, 0) node[right] {$X$};
		\fill (1, 0) circle (1.5pt) node [below] {$x_0$};
		\fill (2, 0) circle (1.5pt) node [below] {$gx_0$};
		\draw (0, 1) -- (2, -1) -- (4.5, 1.5)
			node [right, outer sep = 1ex] 	{$\rho_g(\cdot)$};
		\draw (0, 2) -- (2, 0) -- (4.5, 2.5)
			node [right, outer sep = 1ex] 	{$d_X(gx_0, \cdot)$};
		\end{tikzpicture}
	\caption{The horofunction
		$\rho_g(w) := d_X(gx_0, wx_0) - d_X(x_0, gx_0).$ 
		Fixing every horofunction to be zero at $x_0$ 
		allows us to take the limit of horofunctions $\rho_{g_i}(w)$ 
		as $g_ix_0$ approaches the boundary.}
	\label{horo:R}
	\end{figure}

The intuition in the slim triangles setting is that the horofunction measures 
minus how far you fellow-travel plus how far you then veer off:
\[	\rho_g(w_n) 
	\overset{\text{slim}\Delta\text s}
	\approx - \begin{smallmatrix}
		\text{how far $w_n$}
		\\ \text{follows $g$}
		\end{smallmatrix}
	+ \begin{smallmatrix}
		\text{how far $w_n$}
		\\ \text{veers off}
		\end{smallmatrix}
	.
	\]
See Figure \ref{horo:slim}.
\begin{figure}
	\begin{tikzpicture}[rounded corners = 8]
		\draw[dashed, yscale = 0.5]
			(0, 0) arc (180: 270: 4) node[right] {$\rho_g=0$}
			(0, 0) arc (180: 150: 4)
			;
		\draw[very thick, -{>[scale = 2]}]
			(0, 0) -- node[above]{$-$} 
			(1, 0) -- node[right]{$+$} 
			(1, -2)
			;
		\draw
			(1cm + 2pt, -2) node[below] {$w_nx_0$}
			-- ++(up: 2) 
			-- (6, 0) node[right] {$gx_0$}
			++(up: 2pt)
			-- ++(left: 5)
			-- ++(left: 1) node[left] {$x_0$}
			;
		\end{tikzpicture}
	\caption{When triangles are slim, 
		the horofunction $\rho_g(w_n)$ measures $\approx$
		$-($how far $w_n$ follows $g)$ +  $($how far it veers off).}
	\label{horo:slim}
	\end{figure}
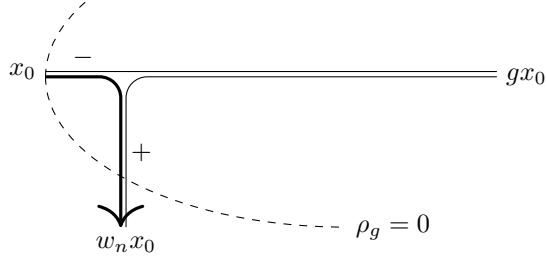

\subsection{\texorpdfstring{$(G, \mu)$-random walk $w_n$ on $(X, d_X, x_0)$}{Random walk}}\label{walk}
A stochastic process in a group~$G$
is a sequence of $G$-valued random variables $w_n.$
We say that $w_n$ is Markov if
$\mathbb P(w_n = g_n \mid w_1 = g_1, \dots, w_{n - 1} = g_{n - 1})$
equals $\mathbb P(w_n = g_n \mid w_{n - 1} = g_{n - 1}).$
In the construction below, the stochastic process $w_n$ in $G$ will be Markov,
but the stochastic process $w_nx_0$ in $X$ will not necessarily be Markov.

Let $X$ be a metric space with basepoint $x_0,$
and let $G \to \Isom X$ be a (not~necessarily injective) group of isometries
with probability measure~$\mu.$
We will call $\mu$ the step distribution,
and the product $(\Omega, \mathbb P) := (G, \mu)^{\mathbb N}$
the step space \cite{tiozzo15}*{\S2.1}.
From there, the $(G, \mu)$-random walk on $(X, d_X, x_0)$
is the stochastic process $(w_n)_{n \in \mathbb N}$
with $w_n: \Omega \to G$ defined to be
\[	(g_1, g_2, g_3, \dots) 
	\longmapsto g_1g_2g_3 \cdots g_n
	\]
for every positive integer $n.$
If called for, we define $w_0 \equiv 1_G.$
Each $w_n$ is referred to as the $n$th (random) location, and
has distribution $\mu^n,$ the $n$-fold convolution:
\[	\mathbb P(w_n = g) 
	= \mu^n(g)
	= \textstyle\sum_{g_1 \cdots g_n = g} 
	\mu(g_1)\cdots \mu(g_n)
	.
	\]
Define the $n$th step to be
the random variable $s_n := w_{n-1}^{-1} w_n,$
which by construction is the $n$th projection 
$s_n: (G, \mu)^{\mathbb N} \to G.$
It follows that the steps are independent and identically distributed 
$G$-valued random variables with law~$\mu.$
Each outcome $(w_n(\omega))_{n \in \mathbb N}$ 
is a sequence of elements of $G,$ 
and is referred to as a sample path.

The map $G^{\mathbb N} \to G^{\mathbb N}$
defined $\omega \mapsto (w_n(\omega))_{n \in \mathbb N}$
induces a pushforward probability measure
on the codomain (range) $G^{\mathbb N}.$
The codomain endowed with this pushforward measure
is called variously the path space, location space,
or Kolmogorov representation space \cite{sawyer95}*{\S2}.
One can alternatively define the path space first
and then take each $w_n$
to be the $n$th projection from the path space~\cite{maher12}*{\S2.1}.

\subsubsection*{Reflected random walk}
Given a group $G$ and a probability measure $\mu,$
we define the reflected probability measure $\check \mu(g) := \mu(g^{-1}).$
The $(G, \check\mu)$-random walk 
is sometimes called the reflected random walk
$(\check w_n)_{n \in \mathbb N}.$

The bi-infinite $(G, \mu)$-random walk $(w_n)_{n \in \mathbb Z}$
is defined to be the unique bi-infinite sequence of $G$-valued random variables
with $w_0 \equiv 1_G$ 
and steps $s_n := \smash{w_{n-1}^{-1} w_n}$
that are independently and identically distributed
according to $\mu.$
It follows that $w_n$ equals $s_1\cdots s_n$ 
and $\smash{s_0^{-1} \cdots s_{n+1}^{-1}}$
for positive and negative integers $n,$ respectively.
In particular, the sequence $(w_{-n})_{n \in \mathbb N}$ 
is a random walk with
the same distribution as the $(G, \check\mu)$-random walk.

\begin{definition}[$a$-iterated random walk] \label{iter}
Given a random walk $(w_n)_{n \in \mathbb N}$ 
and integer $a>0,$
the $a$-iterated random walk
is the sequence of $G$-valued random variables 
$(w_{ai})_{i \in \mathbb N}$
indexed by integers $i > 0.$
\end{definition}

\begin{definition}[Exponential tail]\label{tail}
The $(G, \mu)$-random walk $w_n$ 
on the metric space $(X, d_X)$
with basepoint $x_0$
is said to have \emph{exponential tail} in $X$
if there exists $\lambda > 0$ such that
\begin{equation}
	\sum\nolimits_{g \in G} \mu(g) e^{\lambda|g|} 
	< \infty
	,
	\label{tail:eq}
	\end{equation}
where $|g|$ denotes $d_X(x_0, gx_0).$
More generally, a real-valued stochastic process $Y_n$
is said to have \emph{uniformly exponential tails} in $\mathbb R$ 
if there exists a single $\lambda > 0$ such that for all~$n,$
\[	\mathbb E\big(e^{\lambda|Y_n|}\big) 
	< \infty
	.
	\]
Note that there need not be a uniform upper bound for all $n.$
\end{definition}

If $Y_n$ has uniformly exponential tails in $\mathbb R,$ then so does $-Y_n.$
If $\mu$ satisfies~\eqref{tail:eq} for some $\lambda > 0$
then so does every convolution power $\mu^n$ for the same $\lambda.$
Indeed, if $\sum_g  \mu(g) e^{\lambda|g|} = \ell$ finite, then
\begin{align*}
		\ell^2
		&= \textstyle\sum_a \mu(a) e^{\lambda|a|}  \ell
	\\	&= \textstyle\sum_{a, b} \mu(a) \mu(b) e^{\lambda|a|} e^{\lambda|b|} 
	\intertext{(recall that $|a|$ and $|b|$ 
		denote $d_X(x_0, ax_0)$ and $d_X(x_0, bx_0),$ respectively)} 
		&\ge \textstyle\sum_{a, b} \mu(a) \mu(b) e^{\lambda|ab|}
	\intertext{(since 
		$d_X(a^{-1}x_0, x_0) + d_X(x_0, bx_0) \ge d_X(a^{-1}x_0, bx_0)$
		by the triangle inequality)}
		&= \textstyle\sum_{a, g} \mu(a) \mu(a^{-1}g) e^{\lambda|g|} 
	\\	&= \textstyle\sum_g \mu^2(g) e^{\lambda|g|} 
		.
	\end{align*}
Thus $Y_n = |w_n|$ has uniformly exponential tails
if $\mu$ has exponential tail.

\begin{definition}[positive linear progress]\label{drift}
We say that the $(G, \mu)$-random walk $w_n$ 
on the metric space $(X, d_X)$ with basepoint $x_0$
makes \emph{linear progress with exponential decay} in $X$ 
if there is $C > 0$ so that for all~$n,$
\begin{equation}
	\mathbb P(|w_n| \le n/C) 
	\le Ce^{-n/C}
	,
	\label{drift:eq0}
	\end{equation}
makes \emph{positive linear progress} 
(or has \emph{positive drift}) in $X$~if
\begin{equation}
	\liminf_{n \to \infty} |w_n|/n 
	> 0
	\text{ almost surely},
	\label{drift:eq1}
	\end{equation}
is \emph{asymptotically probability zero}
on bounded sets in $X$ 
if for every~$r,$
\begin{equation}
	\mathbb P(|w_n| \le r) 
	\overset{_n\,}
	\to 0
	\label{drift:eq2}
	\end{equation}
\cite{maher-tiozzo}*{\S5.2,\,p.\,40},
and is \emph{not positive-recurrent} on bounded sets in $X$~if
\begin{equation}
	\mathbb E(|w_n|) 
	\overset{_n\,}
	\to \infty
	,
	\label{drift:eq3}
	\end{equation}
where $|w_n|$ denotes $d_X(x_0, w_nx_0).$
Recall that a Markov chain is either 
transient, positive-recurrent, or neither (null-recurrent)
\cite{woess00}*{I.1.B}.
For linear progress with exponential decay,
it suffices to show that there exist $c_i > 0$ so that
for all but finitely many $n,$
the probability $\mathbb P(|w_n| \le n/c_3) \le$
a finite sum $\sum_j c_{2, j} e^{-n/c_{1, j}}.$

More generally, a (not necessarily Markov) sequence 
of real-valued random variables $Z_n$
makes linear progress with exponential decay in $\mathbb R$
if there is $C > 0$ so that for all $n,$
\(	\mathbb P(Z_n \le n/C) 
	\le Ce^{-n/C}
	\);
see equation~\eqref{intro:drift}.
Properties~\eqref{drift:eq1}--\eqref{drift:eq3} 
generalize to $Z_n$ in $\mathbb R$
by replacing $|w_n|$ with $|Z_n|.$
\end{definition}

\begin{remark}\label{drift:lem}
Equations~\eqref{drift:eq0}--\eqref{drift:eq3} 
are in descending order of strength:
\[	\begin{smallmatrix}
		\text{linear progress with}
		\\ \text{exponential decay}
		\end{smallmatrix}
	\implies
	\begin{smallmatrix}
		\text{positive linear}
		\\ \text{progress}
		\end{smallmatrix}
	\implies
	\begin{smallmatrix}
		\text{asymptotically}
		\\ \text{probability zero}
		\end{smallmatrix}
	\implies
	\begin{smallmatrix}
		\text{not positive-}
		\\ \text{recurrent.}
		\end{smallmatrix}
	\]

The first claim is that linear progress with exponential decay
is indeed a special case of positive linear progress.
For any sequence of events $A_n,$
the Borel-Cantelli Lemma
states that $\sum{\!_n\,} \mathbb P(A_n) < \infty$
implies $\mathbb P(\limsup{\!_n\,} A_n) = 0.$
Exponential decay implies 
\(\sum_n \mathbb P(|Z_n| \le n/C) 
	\)
is finite, and so by Borel-Cantelli,
$|Z_n| \le n/C$ for only finitely many $n$ almost surely.
Thus, $\liminf_n \frac1n |Z_n| \ge 1/C$ almost surely.

The second claim is that positive linear progress in turn implies 
bounded sets have zero asymptotic probability.
To show the contrapositive,
assume there exists an $r$ such that 
the sequence of events $A_n := \{|Z_n| \le r\}$
has probabilities uniformly bounded away from zero
$\liminf_n \mathbb P(A_n) > 0.$
Then there is an $\epsilon > 0$ and a subsequence $n_k$
such that $\inf_k \mathbb P(A_{n_k}) \ge \epsilon.$
It follows that $\mathbb P(\limsup{}\!_k A_{n_k}) \ge \epsilon,$
i.e., the subsequence of events $A_{n_k}$
occurs infinitely often with probability $\ge \epsilon.$
But then so does the full sequence:
$\mathbb P (\limsup A_n) \ge \epsilon.$
The event $\limsup A_n$ implies (is contained in) 
the event $\liminf_n \frac1n|Z_n| = 0,$
which thus also occurs with probability $\ge \epsilon.$
In other words, 
it is not the case that $\liminf_n \frac1n|Z_n| > 0$ almost surely.

The last claim is that $Z_n$ cannot be 
both positive-recurrent and asymptotically probability zero 
on bounded subsets in $\mathbb R.$
Suppose $Z_n$ is positive-recurrent: 
$\sup_n \mathbb E(|Z_n|) \le L$ finite.
By Markov's inequality, $\mathbb P(|Z_n| < 3L) \ge 2/3.$
Thus $\mathbb P(|Z_n| < r)$ does not $\smash{\overset {_n\,} \to 0}$
for every $r$.
\end{remark}

\subsection{Shadows}
Consider the metric space $(X, d_X)$
and the group of isometries $G \to \Isom X.$
Fix a basepoint $x_0 \in X$ and put $|g| = d_X(x_0, gx_0).$
For each $g$ in $G$ and real number $r,$
the shadow about $gx_0$ of radius $r$ is defined to be the set
\begin{align*}
	Shad(g, r) 
	&:= \{w \in G: (w, g)_1 \ge |g| - r\}
	.
	\notag
	\intertext{For each real number $d,$
		we define the shadow about $gx_0$ of depth $d$ to be the set
		}
	S(g, d) 
	&:= \{w \in G: (w, g)_1 \ge d\}
	;
	\end{align*}
in other words, $w \in S(g, d)$ if and only if $(w, g)_1 \ge d.$
Note that it is possible to define shadows 
with basepoints other than $x_0,$
but we will not need such shadows.
Most sources define shadows in terms of radius $Shad(g, r).$
However, following Maher \cite{maher12}*{\S2.3},
we will always define shadows in terms of depth $S(g, d).$
The two definitions are equivalent $S(g, d) = Shad(g, |g| - d),$
and shadows in terms of depth will be
more convenient for our purposes,
including Definition~\ref{s} below.

By the triangle inequality~\eqref{triangle}, 
if $d$ is negative or zero, then $S(g, d) = G,$ 
and if $d$ exceeds $|g|,$ then $S(g, d)$ is empty.

\begin{definition}[Shadow decay]\label{s}
We say that shadows in $X$ decay in $d$
if both $\mu^n(S(g, d))$ 
and $\smash{\check\mu^n(S(g, d)) \overset{_d\,}\to 0}$
in the sense that
\[	\adjustlimits \lim_{d \to \infty} \sup_{n, g} 
	\, \mathbb P(w_n \in S(g, d)) =
	\adjustlimits \lim_{d \to \infty} \sup_{n, g} 
	\, \mathbb P(w_n^{-1} \in S(g, d)) = 0
	,
	\]
where we take each supremum over all 
$n \in \mathbb N$ and $g \in G.$
In other words
there is a function $f$ that decays $f(d) \to 0$ as $d \to \infty$
such that $\mu^n(S(g, d)) \le f(d)$ for all $n, g,$
and there is an analogous function $f'$ for $\check \mu.$
Note that if $f, f'$ differ, we may replace them with $\max\{f, f'\}.$
However, the optimal $f$ for $\mu$ and $f'$ for $\check\mu$ may differ.
See Maher and Tiozzo~\cite{maher-tiozzo}*{Cor\-ol\-lary 5.3}.
\end{definition}

It is known 
\citelist{\cite{maher12}*{Lemma 2.10} 
	\cite{maher-tiozzo}*{Equation 16}}
that when the support of $\mu$ is bounded in $X,$
shadows decay exponentially in $d,$
meaning that there is some constant $C > 0$
so that for every $d$ and $n,$
$\mu^n(S(g, d)) \le Ce^{-d/C}.$
Crucially, $C$ does not depend on $n$ (or $d$).\footnote{
	Exponential decay of shadows in $d$ for fixed $n$
	follows immediately from exponential tail for $\mu.$
	Indeed, take $\lambda > 0$ 
	so that $\sum_g \mu(g)e^{\lambda|g|}$ is finite.
	Then
	\[\smash{
		\mathbb P((w_n, g)_1 \ge d/\lambda)
		\le \mathbb P(|w_n| \ge d/\lambda)
		= \mathbb P(|w_n| \ge d/\lambda)
		= \mathbb P(e^{\lambda|w_n|} \ge e^d)
		\le \mathbb E(e^{\lambda|w_n|}) e^{-d}
		}\]
	where the first inequality follows from 
	$(w_n, g)_1 \le |w_n|$ (from the triangle inequality)
	and the last inequality follows from Markov's Inequality.
	}
However, the proof in this paper will not use exponential decay of shadows.

\section{Proofs}
The results needed for Theorem \ref{thm} 
are organized as follows.
Section \ref{nel:sec} cites the convergence theorem for 
non-elementary weakly hyperbolic random walks
and extracts positive linear progress and shadow decay.

Section \ref{back:sec} 
uses these two properties along with hyperbolicity 
to prove a stronger shadow decay result
(uniform shadow decay).
Section \ref{horo:sec} 
uses this shadow decay result along with exponential tails
to prove a property concerning horofunctions
(uniformly positive horofunctions).

Section~\ref{prog:sec}
uses that property
to derive linear progress with exponential decay
for the $a$-iterated random walk
by way of a progress criterium
(uniformly positive progress).
Lastly, Sec\-tion \ref{iter:sec}
uses exponential tails again to pass from 
the $a$-iter\-ated random walk to the full random walk.
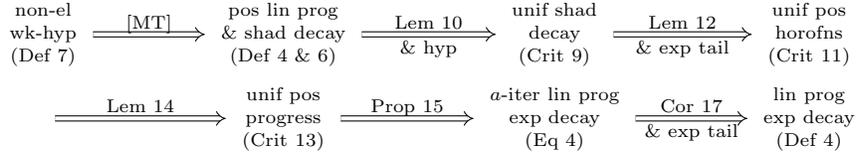
\[\begin{tikzcd}[row sep = 0mm, column sep = 15mm]
		\stackthree{non-el}{wk-hyp}{(Def \ref{nel:def})}
		\rar[Rightarrow]{\smash{\text{\cite{maher-tiozzo}}}}&
		\stackthree{pos lin prog}{\& shad decay}{(Def \ref{drift} \& \ref{s})}
		\rar[Rightarrow]{\text{Lem \ref{back:lem}}}[swap]{\text{\& hyp}}&
		\stackthree{unif shad}{decay}{(Crit \ref{back:def})}
		\rar[Rightarrow]{\text{Lem \ref{horo:lem}}}
			[swap]{\text{\& exp tail}}&
		\stackthree{unif pos}{horofns}{(Crit \ref{horo:def})}
	\\	{}\rar[Rightarrow]{\text{Lem \ref{prog:cor}}}&
		\stackthree{unif pos}{progress}{(Crit \ref{prog:def})}
		\rar[Rightarrow]{\smash{\text{Prop \ref{prog:lem}}}}&
		\stackthree{$a$-iter lin prog}{exp decay}{(Eq \ref{thm:iter})}
		\rar[Rightarrow]{\text{Cor \ref{iter:cor}}}
			[swap]{\text{\& exp tail}}&
		\stackthree{lin prog}{exp decay}{(Def \ref{drift})}
	\end{tikzcd}\]
Our use of metric hyperbolicity
is confined to Sections \ref{nel:sec} and \ref{back:sec}.

\subsection{Convergence}\label{nel:sec}
We cite the convergence of
non-elementary, weakly-hyperbolic random walks
to obtain positive linear progress and shadow decay.

\begin{definition}[Non-elementary, weakly hyperbolic]\label{nel:def}
Let $X$ be a metric spa\-ce, not necessarily locally compact.
A countable group of isometries $G \to \Isom X$ is called
\begin{itemize}
	\item \emph{weakly hyperbolic} if
	$X$ is separable, geodesic, and $\delta$-hyperbolic, and
	\item \emph{non-elementary} if
	$G$ contains two independent loxodromics,
	\end{itemize}
where loxodromics are called independent 
if they have disjoint fixed point sets 
on $\partial X,$ the Gromov boundary 
\citelist{\cite{maher-tiozzo}*{\S1} \cite{osin16}*{Theorem 1.1}}.
The $(G, \mu)$-random walk is called non-elementary 
if $\langle\supp(\mu)\rangle$ is non-elementary.
Our definition of non-elementary
follows Maher and Tiozzo \cite{maher-tiozzo}*{\S1}
and excludes the quasi-parabolic case 
(see Section~\ref{metric}, page~\pageref{classify}).
\end{definition}

\begin{theorem}[\cite{maher-tiozzo}*{Theorems 1.1, 1.2 and Corollary 5.3}]\label{nel:thm}
If the random walk $w_n$ on $X$
is non-elementary and weakly hyperbolic (Definition \ref{nel:def}),
then $w_nx_0$ converges to the Gromov boundary almost surely,
and the associated hitting measure is non-atomic.
In particular, 
$w_n$ makes positive linear progress in $X$ (Definition \ref{drift})
and shadows in $X$ decay in $d$ (Definition \ref{s}).
\end{theorem}
Note that we do not require any moment or tail conditions on $\mu,$
we do not require $X$ to be locally compact,
and we do not require the action of $G$ to be acylindrical.

\subsection{Shadows}\label{back:sec}
In this section we use convergence to the boundary
(specifically, positive linear progress and shadow decay)
along with hyperbolicity
to prove a stronger shadow decay result.

In general, $2(g,w)_1$ equals 
$d_X(gx_0, x_0) + d_X(x_0, wx_0) - d_X(gx_0, wx_0),$
the difference between distance and displacement
when one travels along a concatenation of geodesics
$[gx_0, x_0] \cup [x_0, wx_0].$
Thus the Gromov product $(g,w)_1$ 
measures the inefficiency
in traveling $gx_0$ to $x_0$ to $wx_0,$
and so shadow decay controls the inefficiency
in a random walk as it escapes to infinity,
even when $X$ is not hyperbolic.

If $X$ is hyperbolic,
then this inefficiency looks like a ``backtrack."
In this section, we use hyperbolicity
to prove the two ``single backtracks" in Figure \ref{back:spse}
imply the one ``double backtrack" in Figure \ref{back:claim}.

After this section, 
we will not use hyperbolicity for any of the remaining proofs.

\begin{criterion}[Uniform shadow decay]\label{back:def}
We say that shadows in $X$ decay uniformly
if $\smash{\mathbb P((w_kh, g)_1 \ge d) \overset{_{d,k}}\longrightarrow 0}$
in the sense that
for all $\epsilon > 0,$
for every depth $d\ge$ some $d'(\epsilon),$
and for every time $k\ge$ some $k'(d),$
\begin{equation}
	\sup_{g, h \in G} 
	\mathbb P(w_kh \in S(g, d))
	< \epsilon
	.
	\label{back:eq}
	\end{equation}
\end{criterion}

Note that for $n \ge k,$ 
the random variables $w_k$ and $\kn$
are independent, and so 
$\mathbb P(w_kh \in S(g, d))$
is equal to
$\mathbb P(w_n \in S(g, d) \mid \kn = h).$
Hence equation~\eqref{back:eq} means that
shadows decay independently from the value of $\kn,$
and in particular, independently from 
the distance $|\kn|.$
For $n$ much larger than $k,$
this distance $|\kn|$
approximates progress~$|w_n|.$

\begin{lemma}\label{back:lem}
Consider the $(G, \mu)$-random walk $w_n$
on the metric space $X$
with basepoint $x_0.$
Suppose $X$ is $\delta$-hyperbolic,
bounded subsets have asympototic probability zero 
(Definition \ref{drift}), and
shadows decay (Definition \ref{s}).
Then shadows decay uniformly
(Criterion~\ref{back:def}).
\end{lemma}

Note that positive linear progress implies 
that bounded subsets have asympototic probability zero
(Remark~\ref{drift:lem}).
We do not assume the action of $G$ to be acylindrical.

\begin{proof}[Proof. Hyperbolic geometry step.]
Recall that we write $(w, g)_1$ to denote
\(	\textstyle \frac12 |w|
	+ \frac12 |g|
	- \frac12 d_X(wx_0, gx_0)
	.
	\)
As depicted in Figure \ref{back:spse},
suppose we are given $g, w, h$ in $G$
and $d > 0$ such that
\begin{align}
	|w| &> 2d,
	\label{back:w}
	\\ (w, g)_1 &< d - \delta, \text{ and}
	\label{back:g}
	\\ (\smash{w^{-1}}, h)_1 &< d
	.
	\label{back:h}
	\end{align}
As in Figure \ref{back:claim},
we claim that $(wh, g)_1 < d.$

By $G$-equivariance of the Gromov product,
$(w^{-1}, h)_1$ equals $(1, wh)_w$ 
and so~\eqref{back:h} indeed controls 
the backtrack at the bottom in Figure~\ref{back:spse}.
Note that equations~\eqref{back:g} and~\eqref{back:h}
are equivalent to $w \not\in S(g, d - \delta)$
and $w^{-1} \not\in S(h, d),$ respectively,
and that the claim is equivalent to 
$wh \not\in S(g, d).$

\begin{figure}
	\begin{tikzpicture}[rounded corners = 8]
		\fill[pattern = north east lines, rounded corners = 0] (0, -4pt) 
			rectangle 
			node[above, outer sep = 4pt] 
			{$\overset{\eqref{back:g}} < d - \delta$} 
			(2, 2pt) 
			;
		\draw (0, 0) node[left] {$x_0$} -- (6, 0) 
			node[right] {$gx_0$};
		\fill[pattern = north west lines, rounded corners = 0, yshift = -2pt] 
			(2cm - 2pt, -2) 
			rectangle 
			node[right, outer sep = 2pt] 
			{$\overset{\eqref{back:h}} < d$} 
			(2cm + 4pt, -1) 
			;
		\draw[yshift = -2] 
			(0, 0) -- (2, 0) -- 
			node[left, pos = 0.3] 
			{$|w|\overset{\eqref{back:w}} > 2d \,$} 
			(2, -2) 
			node[below] {$wx_0$};
		\draw[yshift = -2, xshift = 2] 
			(2, -2) -- (2,-1) -- (4,-1) 
			node[right] {$whx_0$};
		\end{tikzpicture}
	\caption{In the geometric step 
		of the proof of Lemma \ref{back:lem},
		we assume shadow decay 
		$(\ref{back:g},\ref{back:h})$
		and positive linear progress \eqref{back:w}.
		Figure \ref{back:claim} depicts the corresponding claim.}
	\label{back:spse}
	\end{figure}
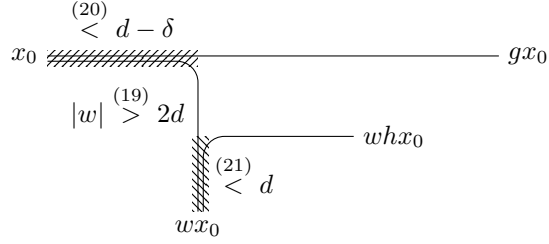

We want to bound how far 
$wh$ and $g$ can fellow-travel.
By the Gromov four-point condition for hyperbolicity
(Figure~\ref{4pt:fig}), 
$wh$ cannot fellow-travel both $w$ and $g$ 
beyond distance $(w, g)_1 + \delta.$
That distance is in turn 
bounded $(w, g)_1 + \delta < d$ by assumption~\eqref{back:g}.
Putting the two together, we have that
\begin{align}
	(wh, w)_1 
	\text{ or } 
	(wh, g)_1 
	\overset {\eqref{4ptG}}\le 
	(w, g)_1 + \delta
	\overset {\eqref{back:g}} < d
	.
	\label{back:or}
	\end{align}
Therefore if we can show $(wh, w)_1 > d,$
then $(wh, g)_1$ must be $< d.$
But by our other two assumptions
$(\ref{back:w}, \ref{back:h})$
and the definition of Gromov product,
\begin{equation}
	\begin{aligned}
			d  \overset{\text{$(\ref{back:w}, \ref{back:h})$}}
			< {}& |w| - (w^{-1},h)_1 
		\\	= {}& |w| - \textstyle\frac12|w| - \frac12|h| + \frac12|wh| 
		\\	= {}& \textstyle\frac12|w| - \frac12|h| + \frac12|wh| 
		\\	= {}& (wh, w)_1
			.
		\end{aligned}
	\label{back:d}
	\end{equation}
Therefore $(wh, g)_1 < d,$ as claimed.

In words, the calculation in~\eqref{back:d}
shows that $w$ is too long~\eqref{back:w}
for the subsequent short backtrack~\eqref{back:h} to cancel out,
and thus $wh$ fellow-travels $w$ instead of $g.$
See also Figure~\ref{back:tree}.
	
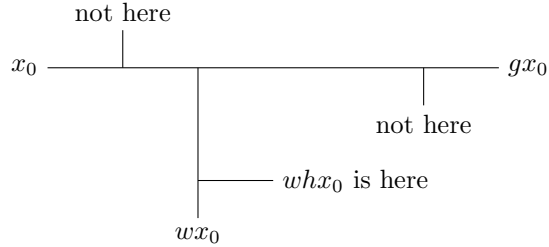
\begin{figure}
	\begin{tikzpicture}
		\draw (0, 0) node[left] {$x_0$} -- (6,0) 
		node[right] {$gx_0$};
		\draw (2, 0) -- (2, -2) node[below] {$wx_0$};
		\draw (2, -1.5) -- (3, -1.5) 
		node[right] {$whx_0$ is here};
		\draw (1, 0) -- (1, 0.5) node[above] {not here};
			\draw (5, 0) -- (5, -0.5) node[below] {not here};
		\end{tikzpicture}
	\caption{Equation \eqref{back:or} means in words that 
		the geodesics between four points in a hyperbolic space 
		look approximately like a tree in one of 
		three possible configurations, 
		and so showing that $wh$ branches off $w$ 
		rules out the case where $wh$ branches off~$g.$}
	\label{back:tree}
	\end{figure}

\subsubsection*{Probabilistic step}
Let $\epsilon > 0.$ We claim that we can find 
sufficiently large real $d > 0$
and integer $k > 0$
so that all three equations~\eqref{back:w}--\eqref{back:h}
hold simultaneously
for $w = w_k$ for all $g, h$ in $G$ 
with probability $> 1 - \epsilon.$

It turns out we must first choose $d$ 
and then choose $k$ depending on $d$ 
(hence the order of quantifiers in Criterion \ref{back:def}).
By decay of shadows (Definition \ref{s}),
we can find a large real $d' > 0$ 
(depending only on $\mu$ and $\epsilon$) 
so that for all real $d \ge d',$ integer $k > 0,$ and $g, h$ in $G,$
the two events
\begin{align}
	\{(w_k, g)_1 < d - \delta\}
	\label{back:g2}
	\\
	\{(w_k^{-1}, h)_1 < d \} 
	\label{back:h2}
	\end{align}
each occur with probability greater than~$1 - \epsilon/3.$
These two events~\eqref{back:g2} and~\eqref{back:h2} each occur 
when $w_k$ and $w_k^{-1}$ 
avoid $S(g, d - \delta)$ and $S(h, d),$ respectively.
Fix arbitrary $d \ge d'.$
Now using our zero asymptotic probability hypothesis (Definition \ref{drift}),
we can find a large integer $k' > 0$ (depending on $d$) 
so that for all integers $k \ge k',$ the~event
\begin{align*}
	A &:= \{|w_k| > 2d\}
	\intertext{has $\mathbb P(A) > 1 - \epsilon/3.$
		Fix arbitrary $k\ge k'.$
		Since~\eqref{back:g2} and~\eqref{back:h2} 
		hold for all $n \in \mathbb N$ 
		and in particular for $n = k,$
		therefore for all $g, h$ in $G,$
		}
	B &:= \{ (w_k, g)_1 < d - \delta \}
	\\ C &:= \{ (w_k^{-1}, h)_1 < d \}
	\end{align*}
have $\mathbb P(B) > 1 - \epsilon/3$
and $\mathbb P(C) > 1 - \epsilon/3.$ 
It follows (regardless of dependence) that 
the intersection of these events occurs 
with probability $\mathbb P(A \cap B \cap C) > 1 - \epsilon,$ as claimed.
Crucially, $d$ and $k$ do not depend on $g$ or $h.$

\subsubsection*{Final step}
For all $g, h$ in $G,$ the Gromov product $(w_kh, g)_1 < d$ 
whenever the three events $A,$ $B,$ $C$ all occur 
(Geometric step), 
which is more than $1 - \epsilon$ of the time 
(Probabilistic step).
Therefore $w_kh \not\in S(g, d)$ 
more than $1-\epsilon$ of the time, 
and so $\mathbb P(w_kh \in S(g, d)) < \epsilon,$ 
as desired.
\end{proof}

\subsection{Horofunctions}\label{horo:sec}
In this section, we use the result of the previous one 
(Criterium \ref{back:def})
to prove a result concerning horofunctions
(Criterium \ref{horo:def}).
We do not use hyperbolicity 
in these remaining Sections~\ref{horo:sec}--\ref{iter:sec}.

\begin{criterion}[Uniformly positive horofunctions]\label{horo:def}
The $(G, \mu)$-random walk $w_n$ on $(X, d_X, x_0)$ 
is said to have uniformly positive horofunctions by time $n_0$ in $X$
if for all $n \ge n_0$ and all positive $t \le$ some $t_0(n),$
\[	\sup\nolimits_{g \in G} 
	\big(\mathbb E \big(
		e^{-t\rho_g(w_n)}
		\big)\big) 
	< 1
	.
	\]
In other words, $\mathbb E (e^{-t\rho_g(w_n)}) \le 1 - \epsilon$
where $\epsilon > 0$ depends on $n$ and $t$ but not $g.$
\end{criterion}

\begin{lemma}[cf.\ \cite{mathieu-sisto}*{9.1}]\label{horo:lem}
Suppose the $(G, \mu)$-random walk $w_n$ 
on $(X, d_X, x_0)$
has exponential tail 
and is not positive-recurrent (Definitions \ref{tail} and \ref{drift}).
Then uniform shadow decay implies uniformly positive horofunctions
(Criteria \ref{back:def} and~\ref{horo:def}).
\end{lemma}
	
We do not require the action of $G$ to be acylindrical,
nor do we require $X$ to be a priori hyperbolic.
However, uniform shadow decay
is a consequence of hyperbolicity by Lemma~\ref{back:lem}.

The following proof is adapted from 
Mathieu and Sisto~\cite{mathieu-sisto}*{Theorem 9.1}.

\begin{proof}[Proof. Step Zero.]
It will help guide the proof to think of 
the goal as being to show 
\[	\text{``Goal":}
	\quad
	\mathbb E(\rho_g(w_n)) 
	> 0
	.
	\]
Indeed, if the horofounction $\rho_g(w_n)$ 
for a given $g$ in $G$ and $n$ in $\mathbb N$
has positive expectation, then
\[	{\textstyle \frac d{dt}} \,
	\mathbb E\big(e^{-t\rho_g(w_n)}\big)
	|_{t = 0}
	= \mathbb E(-\rho_g(w_n))
	< 0
	,
	\]
since 
\(	\frac d{dt} \mathbb E(e^{tY})
	|_{t = 0}
	= \mathbb E(Y)
	\)
in general.\footnote{
	For each $g$ in $G$ and integer $n,$
	the real-valued random variable $\rho_g(w_n)$ 
	has a lower bound 
	(depending on $g$ and $n$).
	It follows that $e^{-t\rho_g(w_n)}$
	has finite expectation 
	for all positive~$t,$ inte\-ger~$n,$ and~$g$ in~$G.$
	In other words, 
	we get a one-sided moment generating function for free.
	We also get a one-sided derivative 
	for each moment generating function $\mathbb E(e^{-t\rho_g(w_n)}).$
	
	However, since we want to bound 
	the whole family of moment generating functions,
	it does not suffice to first differentiate
	and then bound the derivative $\mathbb E(\rho_g(w_n)).$
	Instead, we must reverse the order:
	in Step Three (Estimation) below, 
	we first construct an estimate $f(K, N, t)$
	whose ex\-pect\-ation bounds
	$\mathbb E(e^{-t\rho_g(w_n)})$ from above.
	Since we assume $\mu$ has exponential tail,
	this $f(K,N,t)$ will be finite and differentiable with respect to~$t$ 
	in an open neighborhood about zero.
	
	We work with $e^{-t\rho_g(w_n)}$ 
	instead of $e^{t\rho_g(w_n)}$
	so that our estimates hold for $t \ge 0.$
	}
It would then follow that
$\mathbb E(e^{-t\rho_g(w_n)}) < 1$
for sufficiently small $t > 0.$

For arbitrary $g$ and $w_n$ in $G,$
\begin{equation}
	\rho_g(w_n)
	\ge -|w_n|
	\text{ always,}
	\label{horo:bdd2int}
	\end{equation}
since
\(	|w_n| + \rho_g(w_n)
	= d_X(x_0, w_nx_0)
	+ d_X(w_nx_0, gx_0)
	- d_X(gx_0, x_0)
	,
	\)
which is non-negative by the triangle inequality.
Then the strategy for obtaining $\mathbb E(\rho_g(w_n))$ positive
is to name an event that
(1)~is highly likely, and 
(2)~implies $\rho_g(w_n)$ is highly positive.
In the event $w_n$ lies outside some shadow $S(g, d),$
from the definition of horofunction and shadow,
\(	|w_n| - \rho_g(w_n)
	= 2(w_n, g)_1
	< 2d
	.
	\)
So up to some constant~$d,$
\begin{equation}
	\rho_g(w_n)
	\approx |w_n|
	\text{ when }
	w_n \not\in S(g, d)
	.
	\label{horo:bdd3int}
	\end{equation}
The shadow $S(g, d)$ is unlikely to contain $w_n$ for large $d$
under the assumption of shadow decay; however, 
we will need the stronger property 
of uniform shadow decay 
(obtained from Lemma \ref{back:lem})
for the following reason.

Suppose we have a random {\em estimate} 
and {\em event}
so that $\rho_g(w_n) \ge -$({\em estim\-ate}) in general,
and $\rho(w_n) \ge {}$({\em estimate}) in the case that the event occurs.
Then
\begin{align*}
		\rho_g(w_n)
		&= \rho_g(w_n)(1 - \mathbbm1_{\text{\em event}})
		+ \rho_g(w_n)(\mathbbm1_{\text{\em event}})
	\\	&\ge -(\text{\em estimate})(1 - \mathbbm1_{\text{\em event}})
		+ (\text{\em estimate})(\mathbbm1_{\text{\em event}})
	\intertext{where $\mathbbm1_{\text{\em event}}$
		denotes the characteristic function, 
		i.e., the random variable
		which takes the value 1 for outcomes (sample paths) 
		in the {\em event}
		and the value 0 otherwise.
		Then
		}
		\mathbb E(\rho_g(w_n))
		&\ge \mathbb E\big((\text{\em estimate}) 
		(2\mathbbm1_{\text{\em event}} - 1)\big)
	\\	&= \mathbb E((\text{\em estimate}) \,
		\mathbb E\big(2\mathbbm1_{\text{\em event}} - 1 
			\mid \text{\em estimate})\big)
	,
	\end{align*}
since 
\(	\mathbb E(YZ) 
	= \mathbb E(\mathbb E(YZ \mid Y)) 
	= \mathbb E(Y \mathbb E(Z \mid Y))
	\)
for arbitrary random variables $Y$ and $Z.$
Thus we want to show not only that 
$\mathbb P$({\em event}) is close to 1,
but that $\mathbb P$({\em event $\mid$ estimate} = $h$) 
is close to 1 for all~$h.$
The former does not suffice
because the {\em event} 
and {\em estimate} may be correlated---we do not know 
a priori that the {\em event} 
is not somehow less likely 
when the {\em estimate} is large.
The latter will give us just enough independence to establish the lemma.

We will not use $|w_n|$ as our {\em estimate},
since $\mathbb P(w_n \not\in S(g, d) \mid w_n = h)$
equals either 0 or 1 depending on whether or not $h$ is in $S(g, d).$
In other words, 
we cannot bound $\mathbb P(w_n \not\in S(g, d) \mid w_n = h)$
near 1 for all $h \in G$
because the estimate $|w_n|$ 
contains too much information about the event $w_n \not\in S(g, d).$
Hence we pass to an estimate with less information
with the following adjustment.

By the assumption that the $(G, \mu)$-random walk is not positive-recurrent,
the distance $|w_n|$ has much greater expectation than $|w_k|$
for $n \gg k,$ and so
\begin{equation}
	|w_n| \approx |\kn|.
	\label{horo:bdd4int}
	\end{equation}
Putting~\eqref{horo:bdd2int}--\eqref{horo:bdd4int}
together,
the strategy becomes to find $d, k$ so that for all~$n, g,$
\[	\begin{aligned}
			\rho_g(&w_n) \gtrapprox -|\kn|
			\text{ always, and}
		\\	\rho_g(&w_n) \approx \phantom-|\kn| 
			\text{ when } w_n \not\in S(g, d)
			.
		\end{aligned}
	\]
More precisely, 
let $K = |w_k|,$ let $N = |\kn|,$ and let 
$A =\mathbbm1\{w_n \not\in S(g, d)\}.$
Step One (Metric Geometry)
shows in~\eqref{horo:bdd} that
\begin{align*}
	&\text{claim:} 
	\qquad
	\left\{\begin{aligned}
			\rho_g(&w_n) \ge -K - N \text{ always, and}
		\\	\rho_g(&w_n) \ge -K + N - 2d 
			\text{ when } w_n \not\in S(g, d)
			.
		\end{aligned}\right.
	\intertext{Then Step Two (Probability)
		uses uniform shadow decay
		to find a sufficiently large integer~$k$ and real~$d$
		so that~\eqref{horo:100}
		for all $g$ in $G$ and~$n \ge k,$
		}
	&\text{claim:} 
	\qquad\quad
	\mathbb E(1 - A \mid N) < 0.01.
	\intertext{Lastly, having fixed such $d$ and $k,$
		Step Three (Estimation) 
		constructs in~\eqref{horo:f} a random variable $f(K, N, t)$
		that is finite in a neighborhood of $t = 0$
		and finds a constant $n_0 \ge k$ so that
		}
	&\text{claim:} 
	\qquad
	\left\{\begin{aligned}
			&\mathbb E\big(e^{-t\rho_g(w_n)}\big) 
			\le \mathbb E(f(K,N,t)) 
			\text{ for all } t\ge0 \text{ and } g \text{ and } n\ge k,
		\\	&f(K, N, 0) \equiv 1
			\text{ for all } n\ge k \text{, and}
		\\	&\textstyle\frac d{dt} \mathbb E(f(K, N, t))|_{t = 0} < 0
			\text{ for all } n\ge n_0
			.
		\end{aligned}\right.
	\end{align*}
Crucially, $f(K, N, t)$ will not depend on $g.$
As depicted in Figure \ref{horo:mgf},
it will then follow that 
each $n \ge n_0$ has some $t_0(n)$ such that for all positive $t \le t_0,$
the family of moment generating functions is bounded
\(	\sup_g \mathbb E(e^{-t\rho_g(w_n)})
	\le \mathbb E(f(K, N, t)) 
	< 1
	,
	\)
as desired.
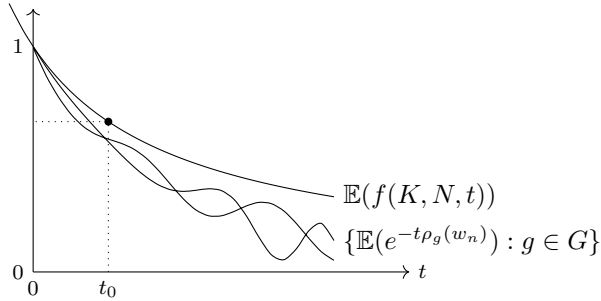
\begin{figure}
	\begin{tikzpicture}
		\draw[->] 
			(0, 0) node[below] {\footnotesize 0} -- 
			(5, 0) node[right] {\footnotesize $t$}
			;
		\draw[->] 
			(0, 0) node[left] {\footnotesize 0} -- (0, 3.5)
			;
		\draw[smooth, domain = 0.28:1]
			plot ({6*\x - 2)}, 1/\x) node[right] {$\mathbb E(f(K,N,t))$};
			;
		\draw[smooth, domain = 0.334:1]
			plot ({6*\x - 2)}, {1/\x - 0.2*(\x*3 - 1)*(2 + sin((\x*3 - 1)^2r*3))}) 
			node[right] {$\{\mathbb E(e^{-t\rho_g(w_n)}) : g \in G\}$}
			plot ({6*\x - 2)}, {1/\x - 0.2*(\x*3 - 1)^0.5*(2 + sin((\x*3 - 1)r*7))})
			;
		\fill
			(0, 3) node[left] {\footnotesize 1}
			(1, 0) node[below] {\footnotesize $t_0$}
			(1, 2) circle (1.5pt)
			;
		\draw[dotted]
			(1, 0) -- (1, 2) -- (0, 2);
		\end{tikzpicture}
	\caption{We bound for all $g$ in $G,$ positive $t,$ and $n \ge k,$
		\(\mathbb E(e^{-t\rho_g(w_n)}) \le \mathbb E(f(K,N,t))\)
		in Step Three (Estimation).
		Moreover, $f(K,N,t)$ is defined without using $g,$
		is finite in a neighborhood of $t = 0,$
		evaluates to one at $t = 0,$ 
		and has derivative $\frac d{dt} \mathbb E(f(K,N,t))|_{t=0}$
		that is negative
		for all $n \ge$ some constant $n_0.$
		It then follows that 
		\(	\textstyle \sup_g
			(\mathbb E(e^{-t\rho_g(w_n)})) 
			\le \mathbb E(f(K,N,t)) < 1\)
		for all $n \ge n_0$
		and all positive $t \le$ some $t_0(n).$
		}
	\label{horo:mgf}
	\end{figure}

%

\subsubsection*{Step One (Metric geometry)}
From the definitions of horofunction and shadow, 
\begin{equation}
	\label{horo:bdd1}
	|w_n| - \rho_g(w_n) 
	= 2(w_n,g)_1
	< 2d
	\quad \text{when } w_n \not\in S(g, d)
	.
	\end{equation}
By the triangle inequality:
\begin{align}
	\label{horo:bdd2}
	|w_n| + \rho_g(w_n) 
	= d_X(x_0, w_nx_0) + d_X(w_nx_0, gx_0) - d_X(gx_0, x_0)
	&\ge 0
	;
\\	\label{horo:bdd3}
	K + |w_n| - N
	= d_X(w_kx_0, x_0) + d_X(x_0, w_nx_0) - d_X(w_nx_0, w_kx_0)
	&\ge 0
	;
\\	\label{horo:bdd4}
	K + N - |w_n|
	= d_X(x_0, w_kx_0) + d_X(w_kx_0, w_nx_0) - d_X(w_nx_0, x_0)
	&\ge 0
	.
	\end{align}
Combining equations~\eqref{horo:bdd1}--\eqref{horo:bdd4}
we get 
\begin{equation}\begin{aligned}
		&\rho_g(w_n) 
		\overset{{^{\scriptstyle\eqref{horo:bdd2}}}}
		\ge -|w_n| \phantom{{} - 2r} 
		\overset{{^{\scriptstyle\eqref{horo:bdd4}}}}
		\ge -K - N\phantom{{} - 2r} 
		\quad
		\text{always, and}
	\\	&\rho_g(w_n)
		\underset{{_{\scriptstyle\eqref{horo:bdd1}}}}
		> \phantom{-{}} |w_n| - 2d
		\underset{{_{\scriptstyle\eqref{horo:bdd3}}}}
		\ge -K + N - 2d
		\quad
		\text{when $w_n \not\in S(g, d).$}
	\end{aligned}\label{horo:bdd}\end{equation}
We will refer to these two bounds $-K - N$ and $-K + N - 2d$ 
as the ``weak" and ``strong" bounds of equation~\eqref{horo:bdd},
respectively.

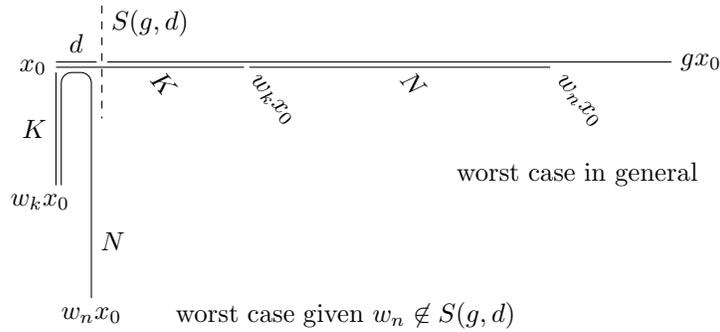
\begin{figure}
	\begin{tikzpicture}[rounded corners = 4]
		\draw 
			(0, 0) node[right] {$gx_0$}
			-- 
			++(left: 7.5) 
			++(left: 2pt) node (shadow) {}
			++(left: 2pt)
			-- node[above] {$d$} 
			++(left: 4mm + 4pt) node (x0) {}
			++(down: 4pt)
			-- node[left] {$K$}
			++(down:1.5) 
			node[below] {$w_kx_0\phantom{w_k}$}
			++(right: 2pt) 
			-- 
			++(up:1.5) 
			-- 
			++(right: 4mm) 
			-- node[right, near end] {$N$} 
			++(down: 3) 
			node[below] (wn1) {$w_nx_0$}
			node[right of = wn1, anchor = west] 
			{worst case given $w_n \not\in S(g, d)$}
			(x0) ++(down: 2pt)  node[left] {$x_0$} 
			-- node[right, rotate = 310] {$K$}
			++(right: 2.5cm) 
			node (wk2) {} 
			node[right, rotate = 310] {$\,w_kx_0$}
			++(right: 2pt)
			-- node[right, rotate = 310] {$N$}
			++(right: 4cm) 
			node {} node[right, rotate = 310] (wn2) {$\,w_nx_0$}
			node[below of = wn2] {worst case in general}
			;
		\node[left of = wn2] {};
		\draw[dashed] 
			(shadow)
			++(up:.75) 
			-- node[right, pos = 0.15] {$S(g, d)$} 
			++(down:1.5)
			;
		\end{tikzpicture}
	\caption{As an aside, these trees are
		configurations that minimize $\rho_g(w_n)$
		in general and in the case $w_n \not\in S(g, d),$ respectively,
		whence an alternate derivation of \eqref{horo:bdd}.
		For example, if the ``worst case in general" geodesic tree depicted 
		connects $x_0,$ $w_kx_0,$ $w_nx_0,$ $gx_0,$
		then $\rho_g(w_n) = -K - N$ by Figure \ref{horo:slim}.
		If no such geodesic tree exists, 
		then $\rho_g(w_n) > -K - N.$
		In particular, we do not assume such geodesic trees exist,
		and we do not assume $X$ is $\delta$-hyperbolic.
		}
	\label{horo:kn}
	\end{figure}

\subsubsection*{Step Two (Probability)}
Fix $k$ and $d$ large enough that 
\[	\mathbb P(w_kh \in S(g, d))
	< 0.01
	\quad \text{ for all }
	g, h \text{ in } G
	\]
using uniform shadow decay, Criterion~\ref{back:def}.
This explicit 0.01 bound for a single pair $k, d$
will be good enough for our estimate in~\eqref{horo:f}.
Since $w_k$ and $\kn$ are independent,
thus $\mathbb P(w_k\kn \in S(g, d) \mid \kn = h) < 0.01$
for all $g, h$ in $G$ and $n \ge k.$
We have defined $A = \mathbbm1\{w_n \not\in S(g, d)\},$
and it follows that
\begin{equation}
	\mathbb E(1 - A \mid \kn)
	< 0.01
	\quad \text{ for all }
	g \text{ in } G 
	\text{ and }
	n \ge k
	.
	\label{horo:100kn}
	\end{equation}
This bound means that
no matter what value $h$
the random variable $\kn$ takes, 
the conditional probability 
$\mathbb P(w_n \in S(g, d) \mid \kn = h) < 0.01.$
Note that the condi\-tional expectation $\mathbb E(1 - A \mid \kn)$
is a real-valued random variable 
taking values in the interval from 0 to 1 
depending on the (random) value of $\kn.$
Equation~\eqref{horo:100kn} bounds 
the entire range of outcomes for this random variable.
We cannot make a similar bound on $\mathbb E(1 - A \mid w_n)$ 
because $w_n$ contains too much information, 
and so conditional expectation 
$\mathbb E(1 - A \mid w_n) = 1 - A = 0$ or 1 
depending on whether or not $w_n \in S(g, d).$

Lastly, since $N$ contains strictly less information than $\kn$
in the sense that $N$ completely depends on $\kn,$ thus
\begin{equation}
	\mathbb E(1 - A \mid N) 
	= \mathbb E(\mathbb E(1 - A \mid \kn) \mid N) 
	< 0.01
	\label{horo:100}
	\end{equation}
for all $g$ in $G$ and $n \ge k.$

\subsubsection*{Step Three (Estimation)}
To obtain the $g$-independent bound in Criterion \ref{horo:def},
we must work directly with the moment generating function
first before passing to the first moment.

Recall that $K = |w_k|,$ $N = |\kn|,$ and $A = \mathbbm1\{w_n \not\in S(g, d)\}.$
We get to apply $-K + N - 2d < \rho_g(w_n)$
the ``strong" bound in equation~\eqref{horo:bdd}
for the sample paths where $A = 1.$
We fall back on the ``weak" bound 
$-K - N \le \rho_g(w_n)$
for the sample paths where $A = 0.$
Hence for all real $t \ge 0,$ integer $n \ge k,$ and $g$ in $G,$
\begin{align*}
		e^{-t\rho_g(w_n)}
		&\overset{\mathclap{\eqref{horo:bdd}}}
		\le e^{-t(-K + N - 2d)} A
		+ e^{-t(-K - N)} (1 - A)
	\\	&= e^{t(K - N + 2d)}
		+ \Big(e^{t(K + N)} - e^{t(K - N + 2d)}\Big) 
		(1 - A)
	\\	&= e^{t(K - N + 2d)}
		+ \Big(e^{tN} - e^{t(-N + 2d)}\Big) e^{tK} 
		(1 - A)
		.
	\end{align*}

This upper bound on $e^{-t\rho_g(w_n)}$ means that
the horofunction $\rho_g(w_n)$ is no less than $-K + N - 2d,$ 
except when $w_n$ lies in $S(g, d),$ 
in which case we must weaken
the bound from $-K + N - 2d$ to $-K - N.$ 

Taking expectation of both sides, it follows that 
for all real $t \ge 0,$ integer $n \ge k,$ and $g$ in $G,$
\begin{equation}\begin{aligned}
	\mathbb E\big(e^{-t\rho_g(w_n)}\big)
	\le \mathbb E\Big(
		e^{t(K - N + 2d)}
		+ \Big(e^{tN} - e^{t(-N + 2d)}\Big)e^{tK}(1 - A)
		\Big)&
	\\ = \mathbb E\Big(
		e^{t(K - N + 2d)}
		+ \Big(e^{tN} - e^{t(-N + 2d)}\Big)
		\,\mathbb E\Big(e^{tK}(1 - A) \mid N\Big)
		\Big)&
	\end{aligned}
	\label{horo:exp}
	\end{equation}
since
\(	\mathbb E(\phi(Y)Z) 
	= \mathbb E(\mathbb E(\phi(Y)Z \mid Y))
	= \mathbb E(\phi(Y)\mathbb E(Z \mid Y))
	\)
for any arbitrary function $\phi$ and random variables $Y, Z,$ 
regardless of dependence.
By exponential tail (Definition \ref{tail}), 
each term on the right is finite for $t$ 
in an open neighborhood of zero depending on 
(i.e., not necessarily uniform in)~$n.$
	
We want to apply our Step Two bound~\eqref{horo:100} to 
$\mathbb E(e^{tK}(1 - A) \mid \kn).$
By Cauchy-Schwarz,
the fact that characteristic functions 
equal their squares,
and~\eqref{horo:100}, 
it follows that for all real $t \ge 0,$ integer $n \ge k,$ and $g$ in $G,$
\begin{align}
		\mathbb E\big(e^{tK}(1 - A) \mid N\big) 
		&\overset{\mathclap{\text{C-S}}}\le
		\sqrt{\mathbb E\big(e^{2tK} \mid N\big) 
			\mathbb E\big((1 - A)^2 \mid N\big)}
		\notag
	\\	&\overset{\mathclap{\eqref{horo:100}}}\le
		\sqrt{\mathbb E\big(e^{2tK} \mid N\big)} \, 0.1
		\notag
	\intertext{which, 
		since $K$ and $N$ are independent,
		}
		&= \sqrt{\mathbb E\big(e^{2tK}\big)} \, 0.1
		\notag
	\intertext{which, 
		since $t$ and $K$ are non-negative,
		}
		&\le \mathbb E\big(e^{2tK}\big) 0.1
		.
		\label{horo:10}
	\end{align}
We are looking to bound~\eqref{horo:exp} 
from above using
\(	\mathbb E\big(e^{tK}(1 - A) \mid N\big) 
	\le \mathbb E\big(e^{2tK}\big) \, 0.1
	,
	\)
equation~\eqref{horo:10}.
Doing so requires the relevant term in~\eqref{horo:exp} 
to be non-negative.
Hence we make a technical adjustment,
replacing $e^{tN} - e^{t(-N + 2d)}$ 
with the strictly larger $e^{t(N + 2d)} - e^{t(-N + 2d)}.$
Only the latter is non-negative (for $t \ge 0$) on all sample paths.
Without this adjustment,
the sign and inequality flip for those sample paths
where $N$ is small enough relative to $d$
to make the ``weak" bound $-K - N$ 
stronger than the ``strong" bound $-K + N - 2d.$

Putting it all together, we find that 
for all real $t \ge 0,$ integer $n \ge k,$ and $g$ in $G,$ 
\begin{equation}
	\mathbb E \Big(e^{-t\rho_g(w_n)}\Big)
	\le \mathbb E \Big(
		\underbrace{e^{t(K - N + 2d)}
			+ \Big(e^{t(N + 2d)} - e^{t(-N + 2d)}\Big) e^{2tK} 0.1
			}_{=: f(K,N,t)}
		\Big)
	,
	\label{horo:f}
	\end{equation}
where we can combine in a single expectation
all terms and factors 
by linearity and independence of $K, N$ repsecitvely.
We define the right-hand side of equation~\eqref{horo:f} 
to be our estimate $\mathbb E(f(K,N,t)).$
Crucially, the random variable
$f(K,N,t)$ depends only on $n$ and $t,$ 
and does not depend on $g.$

By construction,
$\mathbb E\big(e^{-t\rho_g(w_n)}\big) \le \mathbb E(f(K,N,t))$
for all $g,$ positive $t,$ and $n\ge k.$
Also by construction,
$f(K, N, 0) \equiv 1$ for all $n \ge k.$
As argued in Step Zero and depicted in Figure \ref{horo:mgf},
all that remains is to find $n_0$ such that
for all $n \ge n_0,$
\[	\text{claim:}
	\quad
	\textstyle \frac d{dt} 
	\mathbb E(f(K, N, t))|_{t = 0} 
	< 0
	.
	\]
By direct computation from~\eqref{horo:f}, the derivative
\begin{align*}
		\textstyle \frac d{dt} 
		\mathbb E(f(K, N, t))|_{t = 0} 
		&= \mathbb E(K - N + 2d + 0.2N)
	\\	&= \mathbb E(K - 0.8N + 2d)
		.
	\end{align*}
Note that for any random variable $X,$
$\frac{\text d}{\text dt} (\mathbb E(e^{tX}))|_{t = 0} = \mathbb E(X).$

Recall that $d, k$ are fixed constants.
Since $\lim_{n \to \infty} \mathbb E(|w_n|)$ is infinite 
by assumption that the $(G, \mu)$-random walk 
is not positive-recurrent~\eqref{drift:eq3},
thus so is $\lim_{n \to \infty} \mathbb E(N).$
Therefore, 
$\frac d{dt} \mathbb E(f(K, N, t))|_{t = 0}$
is negative for all $n \ge$ some sufficiently positive $n_0.$
The existence of such $n_0$ establishes the lemma.
\end{proof}

\subsection{Progress}\label{prog:sec}
In this section, we use the result of the previous one 
to prove linear progress with exponential decay 
for the $a$-iterated random walk.

We proceed in two steps.
First we show that the $a$-iterated random walk
satisfies a progress condition.
Then we show that every (not necessarily Markov) process 
that satisfies this condition
makes linear progress with exponential decay.


\begin{criterion}[Uniformly positive progress]\label{prog:def}
Let $Z_n$ be a (not necessarily Mar\-kov) sequence 
of real-valued random-variables.
We say that $Z_n$  
makes uniformly positive progress in $\mathbb R$
to mean that there exist $b, \epsilon > 0$ 
so that $\mathbb E(e^{-bZ_0}) $ is finite
and for each integer~$n \ge 0,$
\begin{equation}
	\mathbb E\big(e^{-b(Z_{n+1} - Z_n)} \mid Z_n\big) 
	\le 1 - \epsilon
	\quad \text{almost surely}
	,
	\label{prog:eq}
	\end{equation}
where $\epsilon$ is not allowed to depend on $n.$
We say that the $(G, \mu)$-random walk $w_n$
makes uniformly positive progress in $X$
if~\eqref{prog:eq} holds for $Z_n = |w_n|.$
\end{criterion}

The left-hand side of~\eqref{prog:eq}
is an $\mathbb R_{\ge 0}$-valued random variable
taking values depending on the (random) value of $Z_n.$
Specifically, the left-hand side is defined to be
the random variable taking the real value
$\mathbb E(e^{-b(Z_{n+1} - Z_n)} \mid Z_n = \ell)$
for the outcomes where $Z_n = \ell.$
Note that since $-b(Z_{n + 1} - Z_n)$ is not strictly positive,
equation~\eqref{prog:eq} is not an exponential tail statement.

\begin{lemma}[cf.\ \cite{mathieu-sisto}*{9.1}]\label{prog:cor}	
Uniformly positive horofunctions by time $a$ for $w_n$
(Criterion \ref{horo:def})
implies uniformly positive progress for $Z_n = |w_{an}|$
(Criterion \ref{prog:def}),
for the same constant $a.$
\end{lemma}

As suggested by Figures \ref{prog:fig} and \ref{prog:g},
the horofunction measures progress.
This lemma shows that uniformly positive horofunctions
implies uniformly positive prog\-ress (for the $a$-iterated random walk).

Having established this lemma,
linear progress with exponential decay in $X$
for the $a$-iterated random walk will immediately follow,
i.e., there exists $C > 0$ so that for all $n,$
\(	\mathbb P(|w_{an}| \le n/C) 
	\le Ce^{-n/C}
	,
	\)
by Proposition \ref{prog:lem} below.

\begin{proof}
Suppose the $(G, \mu)$-random walk $w_n$ on $(X, d_X, x_0)$
has uniformly positive horofunctions by time $a.$
Then we can fix $a, b, c > 0$ so that for all $g$ in~$G,$
\[	\mathbb E\big(e^{-b\rho_g(w_a)}\big)
	\le e^{-c}
	;
	\]
see Figure~\ref{prog:g}.
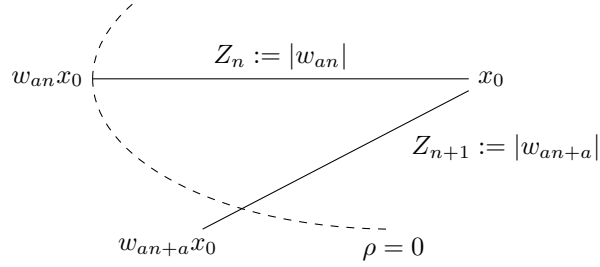
\begin{figure}
	\begin{tikzpicture}
		\draw[dashed, yscale = 0.5]
			(0, 0) arc (180: 270: 4) node[below] {$\rho = 0$}
			(0, 0) arc (180: 150: 4)
			;
		\draw
			(5, 0) node[right] (g) {$x_0$}
			(0, 0) node[left] (x) {$w_{an}x_0$}
			(1, -2) node[below] (w) {$w_{an + a}x_0$}
			(w) -- node[below right, near end] {$Z_{n + 1} := |w_{an + a}|$} (g)
			(x) -- node[above] {$Z_n := |w_{an}|$} (g)
			;
		\end{tikzpicture}
	\caption{Uniformly positive horofunctions means that 
		there exist $a, b, c > 0$ such that for all $g$ in $G,$ 
		the expectation $\mathbb E(e^{-b\rho_g(w_a)}) \le e^{-c}.$
		No matter where you are at time $an,$ 
		you expect to escape further from the basepoint $a$ steps later.
		}
	\label{prog:g}
	\end{figure}
Since $w_a$ 
and $w_{an}^{-1} w_{an + a}$ 
are identically distributed,
\[	\mathbb E\big(e^{-b\rho_g(
		w_{an}^{-1} 
		w_{an + a}
		)}\big)
	\le e^{-c}
	\]
for all integers $n \ge 0$ and $g$ in $G.$
Since $w_{an}^{-1}$ 
and $w_{an}^{-1} w_{an + a}$ 
are independent, 
\[	\mathbb E\big(
		e^{-b\rho_{w_{an}^{-1}}
			(w_{an}^{-1} 
			w_{an + a})} 
		\mid
		w_{an}^{-1}
		\big)
	\le e^{-c}
	\]
for all $n \ge 0.$
Using the definition of horofunction~\eqref{horo:eq}, 
we can then rewrite 
\(	\rho_{w_{an}^{-1}}
	(w_{an}^{-1} w_{an + a})
	\)
as $|w_{an + a}| - |w_{an}|$;
see Figure~\ref{prog:fig}.
Letting $Z_{n} := |w_{an}|,$ 
we find
\begin{equation}
	\mathbb E\big(
		e^{-b(Z_{n + 1} 
			- Z_{n})}
		\mid
		w_{an}^{-1}
		\big)
	\le e^{-c}
	.
	\label{prog:wai}
	\end{equation} 
Since $Z_{n}$ 
completely depends on $w_{an}^{-1},$
we can take conditional expectation on both sides 
and simplify to
\[	\mathbb E\big(
		e^{-b(Z_{n + 1} 
			- Z_{n})}
		\mid
		Z_{n}
		\big)
	= \mathbb E\big( 
		\mathbb E\big(
			e^{-b(Z_{n + 1} 
			- Z_{n})}
			\mid
			w_{an}^{-1}
			\big) 
		\mid
		Z_{n}
		\big)
	\overset{\eqref{prog:wai}}
	\le e^{-c}
	.
	\]
Thus the $Z_{n} := |w_{an}|$ makes 
uniformly positive progress in $\mathbb R$
(Criterion \ref{prog:def}).
\end{proof}

\begin{proposition}\label{prog:lem}
Uniformly positive progress in $\mathbb R$ (Criterion \ref{prog:def})
implies linear prog\-ress with exponential decay in $\mathbb R$ 
(Equation~\ref{intro:drift}).
\end{proposition}

This proposition is a special case of stochastic dominance;
see for example Lindvall \cite{lindvall02}*{Chapter III Theorem 5.8}.
However, we provide a short direct proof of Proposition \ref{prog:lem} below,
for the convenience of the reader.

Recall that although $w_n$ is a Markov process in $G,$
the process $w_nx_0$ in $X$ 
and the process $d_X(x_0, w_nx_0)$ in $\mathbb R$
do not necessarily have the Markov property.
Accordingly, we do not assume the Markov property in this proposition.

\begin{proof}[Proof. Conditional expectation step.]
Suppose $Z_n$ is a (not necessarily Markov) sequence of random-variables
that makes uniformly positive progress in $\mathbb R.$
Then we can find $b, c > 0$ so that
\begin{align}
	\notag &\mathbb E\big(
			e^{-b(Z_{n+1} - Z_n)} 
			\mid
			Z_n
			\big) 
		\le e^{-c}
	\intertext{for each integer $n \ge 0.$
		Since $e^{bZ_n}$ is completely determined by $Z_n,$
		we can pull it out of the expectation,
		so that for all $n,$
		}
	\notag &\mathbb E\big(
			e^{-bZ_{n + 1}} 
			\mid
			Z_n
			\big)
		\le e^{-bZ_n - c}
		.
	\intertext{Taking expectation on both sides, 
		we obtain that for all $n,$
		}
	\label{prog:geo}
		&\mathbb E\big(e^{-bZ_{n + 1}}\big)
		\le \mathbb E\big(e^{-bZ_n}\big)
		e^{-c}
		.
	\end{align}

\subsubsection*{Inductive step}
We claim that $\mathbb E(e^{-bZ_n}) \le Le^{-cn}$ for all $n,$
where $L = \mathbb E(e^{-bZ_0}),$ 
which is finite by assumption (Criterion \ref{prog:def}).
By definition, the claim holds for the base case $n = 0.$
Now assume for induction that 
the claim holds for some given integer $n \ge 0.$
From~\eqref{prog:geo} we have that
\(	\mathbb E(e^{-bZ_{n + 1}})
	\le \mathbb E(e^{-bZ_n}) e^{-c}
	,
	\)
which by our inductive hypothesis 
is $\le Le^{-c(n + 1)}.$
The claim follows by induction.
Rewriting, we have that for each integer $n \ge 0,$
\begin{equation}
	\mathbb E\big(e^{-bZ_n + cn}\big) 
	\le L
	.
	\label{prog:ind}
	\end{equation}

\subsubsection*{Markov inequality step}
Recall that for every
$\mathbb R_{\ge0}$-valued random variable $Y$ 
and real number $\ell,$ 
Markov's inequality gives 
$\mathbb P(Y \ge \ell) \le \mathbb E(Y) \ell^{-1}.$
Hence for all~$n,$
\[	\mathbb P\big(
		Z_n
		\le \textstyle\frac{cn}{2b}
		\big)
	= \mathbb P\big(
		e^{-bZ_n + cn} 
		\ge e^{cn/2}
		\big)
	\overset{\begin{smallmatrix}
		\text{Markov Inequality} \\ \text{and }\eqref{prog:ind}
		\end{smallmatrix}}
	\le Le^{-cn/2}
	,
	\]
and we have shown that the real-valued stochastic process $Z_n$
makes linear progress with exponential decay in $\mathbb R$
(Equation~\ref{intro:drift}).
\end{proof}

\subsection{Remainders}\label{iter:sec}
In this section we prove that
linear progress with exponential decay for the full random walk
follows from exponential tail for $\mu$
and linear progress with exponential decay for the $a$-iterated random walk.
The result is a special case of the following probabilistic fact.

\begin{proposition}\label{iter:lem}
Suppose $Y_n$ has uniformly exponential tail in $\mathbb R$ 
(Definition \ref{tail}),
and suppose $Z_n$ makes linear progress with exponential decay
in $\mathbb R$ (Equation \ref{intro:drift}).
Then the sum $Z_n + Y_n$ 
and the difference $Z_n - Y_n$
both make linear progress with exponential decay as well.
\end{proposition}
Note that we do not make any assumptions of independence.

\begin{proof}
It suffices to show that $-|Y_n| + Z_n$ 
makes linear progress with exponential decay,
since the sum $Y_n + Z_n$ and difference $-Y_n + Z_n$
are both $\ge -|Y_n| + Z_n.$
Fix sufficiently large $C > 0$ 
and sufficiently small $\lambda > 0$ so that
\begin{align*}
	\mathbb P\big(Z_n \le n/C\big) 
	&\le Ce^{-n/C} 
	\text{ and}
	\\ \mathbb E\big(e^{\lambda|Y_n|}\big) 
	&< \infty
	\end{align*}
for all $n.$
By Markov's inequality,
\begin{align*}
		\mathbb P\big(|Y_n| \ge \textstyle\frac n{2C}\big)
		& = \mathbb P\big(e^{\lambda|Y_n|} \ge e^{\lambda n/(2C)}\big)
	\\	& \le \mathbb E\big(e^{\lambda|Y_n|}\big) 
		\, e^{-\lambda n/(2C)}
		.
	\intertext{If one real number is $< \frac n{2C}$
		and another is $> \frac nC,$ 
		then their difference is $> \frac n{2C}.$ 
		By the contrapositive,
		}
		\mathbb P\big(-|Y_n| + Z_n \le \textstyle\frac n{2C}\big)
		&\le \mathbb P\le\big(
			|Y_n| \ge \textstyle\frac n{2C}
			\text{ or } 
			Z_n \le \frac nC 
			\big)
	\\	&\le 
		\mathbb E\big(e^{\lambda|Y_n|}\big) 
		\, e^{-\lambda n/(2C)}
		+
		Ce^{-n/C} 
		.
	\end{align*}
(Note that for arbitrary events, 
$\mathbb P(A \cup B) \le \mathbb P(A) + \mathbb P(B),$
regardless of dependence.)
It follows that $-|Y_n| + Z_n$ 
makes linear progress with exponential decay.
\end{proof}

\begin{corollary}[cf.\ \cite{mathieu-sisto}*{9.1}]\label{iter:cor}
Suppose the random walk $w_n$ on $X$
has exponential tail in $X$ (Definition \ref{tail}).
If the $a$-iterated random walk $w_{ai}$ (Definition \ref{iter})
makes linear progress with exponential decay in $X$ (Definition \ref{drift}),
then so does the full random walk $w_n.$
\end{corollary}

\begin{proof}
Fix the $(G, \mu)$-random walk $w_n$ on $(X, d_X, x_0)$
and $a > 0$ from the hypotheses.
For clarity, we will index the $a$-iterated walk by $i,$
$(w_{ai})_{i \in \mathbb N}.$
For each integer $n > 0,$ write
\[n := ai(n) + r(n)\]
where $i(n)$ and $r(n)$ are non-negative integers depending on $n$
such that $r(n) \le a - 1.$
Using linear progress with exponential decay (indexed by $i$),
fix $C > 0$ so that
\begin{align*}
		\mathbb P(|w_{ai}| \le i/C)
		&\le Ce^{-i/C}
	\intertext{for all~$i,$
		where $|w_{ai}|$ denotes $d_X(x_0, w_{ai}x_0).$
		For all $n \ge a,$ 
		we have that $n \le 2ai(n),$
		and thus $\frac n{2a} \le i(n),$
		and thus 
		}
		\mathbb P\big(
			|w_{ai(n)}| 
			\le \textstyle \frac n{2aC}
			\big)
		&\le C e^{-n/(2aC)}
		.
	\end{align*}
This equation holds for all but finitely many $n,$
and so $Z_n := |w_{ai(n)}|$ makes
linear progress with exponential decay indexed $n.$

Since $\mu$ has exponential tail in $X,$
thus $|w_n|$ has uniformly exponential tail in $X$;
see Definition \ref{tail}.
It follows that
$Y_n := |\smash{w^{-1}_{ai(n)}w_n}|,$
which is distributed identically to $|w_{r(n)}|,$
has uniformly exponential tails as well.

By Proposition~\ref{iter:lem},
the difference $-Y_n + Z_n$
makes linear progress with exponential decay in $\mathbb R.$
Since $|w_n| \ge -Y_n + Z_n$ 
by the triangle inequality, 
hence $|w_n|$ makes 
linear progress with exponential decay in $\mathbb R.$
By definition, 
the $(G, \mu)$-random walk makes
linear progress with exponential decay in $X.$
\end{proof}

\renewcommand{\sectionmark}[1]{\markright{#1}{}}

\renewcommand{\sectionmark}[1]{\markright{\thesection.\ #1}{}}

\end{document}